\documentclass[11pt]{article}
\usepackage[margin=1in]{geometry} 
\usepackage{amssymb,amsfonts,amsmath,bbm,mathrsfs,stmaryrd}
\usepackage{xcolor}
\usepackage{url}
\usepackage{dsfont}

\usepackage{enumerate}

\usepackage[colorlinks,
             linkcolor=black!75!red,
             citecolor=blue,
             pdftitle={rnd},
             pdfauthor={},
             pdfproducer={pdfLaTeX},
             pdfpagemode=None,
             bookmarksopen=true
             bookmarksnumbered=true]{hyperref}

\usepackage{tikz}
\usetikzlibrary{arrows,calc,decorations.pathreplacing,decorations.markings,intersections,shapes.geometric,through,fit,shapes.symbols,positioning,decorations.pathmorphing}

\usepackage{braket}

\usepackage[amsmath,thmmarks,hyperref]{ntheorem}
\usepackage{cleveref}

\creflabelformat{enumi}{#2(#1)#3}

\crefname{enumi}{}{}
\crefformat{enumi}{(#2#1#3)}
\Crefformat{enumi}{(#2#1#3)}

\crefname{section}{Section}{Sections}
\crefformat{section}{#2Section~#1#3} 
\Crefformat{section}{#2Section~#1#3} 

\crefname{subsection}{\S}{\S\S}
\AtBeginDocument{%
  \crefformat{subsection}{#2\S#1#3}%
  \Crefformat{subsection}{#2\S#1#3}%
}

\theoremstyle{plain}

\newtheorem{lemma}{Lemma}[section]
\newtheorem{proposition}[lemma]{Proposition}
\newtheorem{corollary}[lemma]{Corollary}
\newtheorem{theorem}[lemma]{Theorem}

\newtheorem{theoremlet}{Theorem}

\newtheorem{corlet}[theoremlet]{Corollary}
\theoremstyle{nonumberplain}

\theoremstyle{plain}
\theorembodyfont{\upshape}
\theoremsymbol{\ensuremath{\blacklozenge}}

\newtheorem{definition}[lemma]{Definition}

\newtheorem{remark}[lemma]{Remark}
\newtheorem{convention}[lemma]{Convention}
\newtheorem{notation}[lemma]{Notation}

\crefname{definition}{definition}{definitions}
\crefformat{definition}{#2definition~#1#3} 
\Crefformat{definition}{#2Definition~#1#3} 

\crefname{ex}{example}{examples}
\crefformat{example}{#2example~#1#3} 
\Crefformat{example}{#2Example~#1#3} 

\crefname{remark}{remark}{remarks}
\crefformat{remark}{#2remark~#1#3} 
\Crefformat{remark}{#2Remark~#1#3} 

\crefname{convention}{convention}{conventions}
\crefformat{convention}{#2convention~#1#3} 
\Crefformat{convention}{#2Convention~#1#3} 

\crefname{notation}{notation}{notations}
\crefformat{notation}{#2notation~#1#3} 
\Crefformat{notation}{#2Notation~#1#3}

\crefname{lemma}{lemma}{lemmas}
\crefformat{lemma}{#2lemma~#1#3} 
\Crefformat{lemma}{#2Lemma~#1#3} 

\crefname{proposition}{proposition}{propositions}
\crefformat{proposition}{#2proposition~#1#3} 
\Crefformat{proposition}{#2Proposition~#1#3} 

\crefname{corollary}{corollary}{corollaries}
\crefformat{corollary}{#2corollary~#1#3} 
\Crefformat{corollary}{#2Corollary~#1#3} 

\crefname{theorem}{theorem}{theorems}
\crefformat{theorem}{#2theorem~#1#3} 
\Crefformat{theorem}{#2Theorem~#1#3} 

\crefname{assumption}{assumption}{Assumptions}
\crefformat{assumption}{#2assumption~#1#3} 
\Crefformat{assumption}{#2Assumption~#1#3} 

\crefname{equation}{}{}
\crefformat{equation}{(#2#1#3)} 
\Crefformat{equation}{(#2#1#3)}

\numberwithin{equation}{section}

\theoremstyle{nonumberplain}
\theoremsymbol{\ensuremath{\blacksquare}}

\newtheorem{proof}{Proof}

\newcommand\bC{{\mathbb C}}

\newcommand\bN{{\mathbb N}}

\newcommand\bP{{\mathbb P}}

\newcommand\bR{{\mathbb R}}
\newcommand\bS{{\mathbb S}}

\newcommand\bZ{{\mathbb Z}}

\newcommand\cO{{\mathcal O}}
\newcommand\cP{{\mathcal P}}

\newcommand\cX{{\mathcal X}}
\newcommand\cY{{\mathcal Y}}

\newcommand{\op}[1]{\operatorname{#1}}

\DeclareMathOperator{\id}{id}

\newcommand{\cat}[1]{\textsc{#1}}

\newcommand{\qedhere}{\mbox{}\hfill\ensuremath{\blacksquare}}

\title{Random quantum graphs}
\author{Alexandru Chirvasitu and Mateusz Wasilewski}

\begin{document}

\date{}

\newcommand{\Addresses}{{
  \bigskip
  \footnotesize

  \textsc{Department of Mathematics, University at Buffalo, Buffalo,
    NY 14260-2900, USA}\par\nopagebreak \textit{E-mail address}:
  \texttt{achirvas@buffalo.edu}

  \medskip

  \textsc{Department of Mathematics - Section of Analysis, KU Leuven, Leuven,
    Belgium}\par\nopagebreak \textit{E-mail address}:
  \texttt{mateusz.wasilewski@kuleuven.be}

}}

\maketitle

\begin{abstract}
  We prove a number of results to the effect that generic quantum graphs (defined via operator systems as in the work of Duan-Severini-Winter / Weaver) have few symmetries: for a Zariski-dense open set of tuples $(X_1,\cdots,X_d)$ of traceless self-adjoint operators in the $n\times n$ matrix algebra the corresponding operator system has trivial automorphism group, in the largest possible range for the parameters: $2\le d\le n^2-3$. Moreover, the automorphism group is generically abelian in the larger parameter range $1\le d\le n^2-2$. This then implies that for those respective parameters the corresponding random-quantum-graph model built on the GUE ensembles of $X_i$'s (mimicking the Erd\H{o}s-R\'{e}nyi $G(n,p)$ model) has trivial/abelian automorphism group almost surely.
\end{abstract}


\vspace{.5cm}



\section*{Introduction}
The theory of quantum graphs can be traced back to the work of Duan, Severini and Winter (\cite{MR3015725}) and Weaver (\cite{MR2908249}). In the former the authors define a quantum confusability graph associated to a quantum channel, much like a confusability graph arises from a classical channel. In the latter the author develops the theory of quantum relations, inspired by his previous work with Kuperberg (\cite{MR2908248}) on quantum metric spaces. In both cases quantum graphs turn out to be encoded by operator systems.

The work of Weaver provides a unified framework for classical and quantum relations. In particular, one can identify classical graphs with operator systems of a specific kind. As proved in \cite[Theorem 3.3]{MR3337337}, these operator systems are actually complete invariants of the graphs; two graphs are isomorphic if and only if the associated operator systems are unitally completely order isomorphic. This allowed the authors to associate new -- \emph{quantum} -- invariants to classical graphs.

Another approach to quantum graphs was developed in \cite{MuRuVe18}. The authors, inspired by previous work on quantum isomorphisms of graphs \cite{MR4097284}, developed a categorical framework for quantum functions. What turned out from their investigations is that most of the information about the quantum isomorphisms can be recovered from appropriate categories, apart from commutativity of the algebra of functions on the vertices. This led them to include quantum graphs into their considerations, which they define to be finite dimensional $C^{\ast}$-algebras equipped with a noncommutative analogue of an adjacency matrix. A priori there is no reason for this notion of a quantum graph to be the same as the one mentioned above, but both approaches turn out to be essentially equivalent (cf. \Cref{Subsec:quantgraph}).

In this article we start studying random quantum graphs; random classical graphs have a long history (see \cite{MR120167}). One of the first considered issues was, whether a typical graph admits non-trivial symmetries. Symmetric objects are expected to be special and it was confirmed in \cite{er} that, as the size of the graph increases, the proportion of graphs admitting a non-trivial automorphism tends to $0$. We will address two different quantum versions of this problem.

The first one has to do with \emph{quantum} automorphisms of graphs. In \cite{MR4097284} the authors discovered an intriguing connection between certain non-local games, originating in quantum information theory, and the quantum automorphism groups of graphs introduced by Banica in \cite{MR2146039} (see also \cite{Bi03}). In \Cref{se:qaut} we improve upon known results about quantum automorphism groups of random graphs, in particular we treat the case of the Erd\H{o}s-R\'{e}nyi $G(n,p)$ model.
\begin{theoremlet}[\Cref{pr.grph-rnd}]
Fix $p\in (0,1)$. Then the quantum automorphism group of a random graph in the $G(n,p)$ model is trivial with overwhelming probability.
\end{theoremlet}
What is more, we treat the case of random regular graphs (\Cref{pr.reg}) and use it to produce large families of isospectral graphs that are not quantum isomorphic (\Cref{pr.g'g''}).

In \Cref{se:gen} we start exploring a different quantum variant, namely automorphism groups of random quantum graphs. The easiest way to provide a random model for a quantum graph is to view it as an operator subsystem $V \subset M_n$ of a matrix algebra -- then $V$ of a fixed dimension can be chosen according to the Haar measure on the Grassmannian. In this way we arrive at the $QG(n,d)$ model. The automorphism group of $V\subset M_n$ consists of those automorphisms of $M_n$ that preserve $V$. Our first result is that the automorphism group is typically abelian, unless $d=0$ or $d=n^2-1$, which correspond to the operator systems $\mathbb{C} \mathds{1}$ and $M_n$, respectively. An important difference from the classical results about Erd\H{o}s-R\'{e}nyi graphs is that in our case the dimension of the quantum graph is kept fixed, we do not need to let it go to infinity. It is mostly a consequence of the fact that the set of quantum graphs of fixed dimension is not discrete, so the outliers can (and do) form a non-trivial measure zero set.

\begin{theoremlet}[\Cref{cor:genab}]\label{Thm:abelaut}
Let $n\in \bN$ and $1\leqslant d \leqslant n^2-2$. Then the automorphism group of a quantum graph in the $QG(n,d)$ model is almost surely abelian.
\end{theoremlet}

Note that we cannot hope for a better result in this regime, because for $d=1$ the random operator system we get is spanned by the identity and one self-adjoint operator, so it would generate an abelian subalgebra, which clearly has a non-trivial automorphism group, for example all the unitaries lying in the algebra itself. It turns out, however, that if we exclude this case (and the corresponding case $d=n^2-2$ on the other edge) then we do get trivial automorphism groups almost surely.

\begin{theoremlet}[\Cref{th:new5}]\label{Thm:trivaut}
Let $n\geqslant 3$ and $2\leqslant d \leqslant n^2-3$. Then the automorphism group of a quantum graph in the $QG(n,d)$ model is almost surely trivial.
\end{theoremlet}

In graph theory there are two main models of random graphs: $G(n,M)$ and $G(n,p)$. Our $QG(n,d)$ is an analogue of $G(n,M)$ and we can also consider the counterpart of $G(n,p)$ denoted by $QG(n,p)$. However, we cannot define this model as is done classically, because quantum graphs notoriously lack an analogue of vertices. But we can easily derive the $QG(n,p)$ as a weighted sum of $QG(n,d)$ models and an elementary computation yields the following corollary. Note that in this case we do come closer to the original statements about Erd\H{o}s-R\'{e}nyi graphs, where the size goes to infinity.

\begin{corlet}[\Cref{Prop:qgnp}]\label{Cor:trivaut}
If $p(n) \in (0,1)$ satisfies $\lim_{n\to \infty} n^2 p(n) = \infty$ and $\lim_{n\to\infty} (1-p(n))n^2 = \infty$ then as $n$ goes to infinity the automorphism group of a quantum graph in the $QG(n,p(n))$ model is trivial asymptotically almost surely.
\end{corlet}

There is another corollary that we would like to mention here. We view quantum graphs as operator subsystems $V\subset M_n$ of matrix algebras and the matrix algebra itself is part of the structure, being a quantum counterpart of the coordinate algebra of the vertex set. But operator systems are interesting in their own right and it is important to see what kind of information about them we can recover. It turns out that in a suitable regime one can obtain $M_n$ from an operator system $V$ chosen according to $QG(n,d)$ as its $C^{\ast}$-envelope, thus we can conclude that these operator systems do not have nontrivial automorphisms.

\begin{corlet}[\Cref{Prop:opsysaut}]\label{Cor:trivautopsys} 
Let $n\geqslant 3$ and $2\leqslant d \leqslant n^2-3$. For almost every operator system $V$ in the $QG(n,d)$ model the only unital complete order self-isomorphism of $V$ is the identity. 
\end{corlet}
\subsection*{Acknowledgements}

We thank Mike Brannan for numerous conversations on the topics covered here, and the anonymous referee for the very careful reading and engaging exchange.

AC is grateful for funding through NSF grant DMS-2001128.

MW was supported by the Research Foundation – Flanders (FWO) through a Postdoctoral Fellowship and by long term structural funding - Methusalem grant of the Flemish Government.

\section{Preliminaries}\label{se.prel}

\subsection{Random structures}

As is customary in the literature, for $p\in [0,1]$ and a positive integer $n$ we denote by $G(n,p)$ the simple random graph built on the Erd\"os-R\'enyi model with probability $p$: $n$ vertices and each of the possible $n\choose{2}$ edges has probability $p$ of being present, independently of each other.

On the other hand, given positive integers $r<n$, we write $G(n,r)$ for the {\it random $r$-regular graph}: an $r$-regular graph on vertex set $[n]=\{1,\cdots,n\}$ with all such graphs being assigned equal probability (as in \cite{bol-dist}, for instance).

We use the following terminology, following, e.g. \cite[Definition 3]{tv-univ}.

\begin{definition}\label{def.hp}
  We say that an $n$-dependent event $E$ holds {\it with overwhelming probability} (or {\it overwhelmingly}) if $P(E)=1-O(n^{-c})$ for every positive constant $c$ (where the scaling factor implicit in the $O$ notation is $c$-dependent).
\end{definition}

\subsection{Quantum graphs - the quantum adjacency matrix and the quantum relations perspective}\label{Subsec:quantgraph}

The aim of this section is to roughly outline how quantum graphs can be seen in two distinct but ultimately equivalent lights.  The following results can essentially be deduced from Section 7 of \cite{MuRuVe18}.

\begin{definition}[\cite{MR2908249}, Definition 2.1]
Let $H$ be a finite dimensional Hilbert space and $M \subseteq B(H)$ be a finite dimensional C$^\ast$-algebra. A {
\it quantum relation on $M$} is an $M'-M'$-bimodule $V \subseteq B(H)$, i.e. a linear subspace $V$ such that $M'VM' \subseteq V$.  We call such a quantum relation $V$ on $M$ {\it symmetric} if $V$ is self-adjoint:  $V = V^*$, and {\it reflexive} if $M' \subset V$.

A symmetric and reflexive quantum relation on $M$ is called a quantum graph.

In this article we will be mostly interested in the case $M=M_n \subset B(\mathbb{C}^n)\simeq M_n$. Then $M' = \mathbb{C} \mathds{1}$, so a quantum graph on $M_n$ will be an operator subsystem  $V\subset M_n$.
\end{definition}

In particular, a reflexive and symmetric quantum relation over any $M$ is an operator system in $B(H)$, and every operator system is a reflexive, symmetric quantum relation on $B(H)$.  The classical motivation for this terminology comes from the following fact:

\begin{proposition}[\cite{MR2908249}, Proposition 2.2]
  Let $X$ be a finite set.  There is a bijective correspondence between  relations $R \subseteq X \times X$ and quantum relations on $\ell^\infty(X) \subseteq B(\ell^2(X))$.  The bijection is given by
\begin{equation*}
  R \mapsto V_R = \text{span}\{e_{ij}: (i,j) \in R\} \qquad V \mapsto R_V = \{(i,j) \in X \times X : \exists T \in V \text{ s.t. } T_{ij} \ne 0\}.
\end{equation*}  
\end{proposition}

Since any graph on a vertex set $X$ with self-loops and no multiple edges corresponds to a reflexive symmetric relation $R \subseteq X \times X$ ($R$ is the set of edges), Weaver's notion of a quantum graph generalizes both the classical finite graphs and matrix quantum graphs coming from Duan-Severini-Winter (quantum confusability graphs, see \cite{MR3015725}).

In the classical setting we have therefore three ways to view a graph: an adjacency matrix, a relation on the set of vertices and an operator system of a special kind. We would like to explain how a similar correspondence works for quantum graphs, where we replace the algebra of functions on the vertex set by a matrix algebra. To make the approach cleaner, we will first discuss the case of general quantum relations and then we will discuss what the symmetry and reflexivity mean in all three pictures.

In order to work with adjacency matrices in the sense of \cite{MuRuVe18}, we have to pick a specific functional on the matrix algebra, and the choice of normalization is not consistent in the literature. We will be working with the pair $(M_n, \tau)$, where $\tau(x):= n \op{Tr}(x)$ so that $\tau(\mathds{1}) = n^2 = \dim M_n$.
\begin{lemma}
Let $m: M_n \otimes M_n \to M_n$ be the multiplication map. Let $m^{\ast}: M_n \to M_n \otimes M_n$ be the adjoint with respect to the inner product induced by $\tau$. Then $m^{\ast}(e_{ij}) = \frac{1}{n} \sum_{k} e_{ik} \otimes e_{kj}$. It follows that $m m^{\ast} = Id$.
\end{lemma}
\begin{proof}
Direct verification.
\end{proof}

We will first handle the correspondence between subspaces of $M_n$ (quantum relations in the sense of Weaver) and projections in $M_n \otimes M_n^{op}$. 
\begin{proposition}[\cite{MR2908249}, Proposition 2.23]
There is a bijective correspondence between the set of subspaces of $M_n$ and projections in $M_n \otimes M_n^{op}$.
\end{proposition}
\begin{proof}
  Recall that $M_n \otimes M_n^{op}$ acts on $M_n$ by left-right multiplication. Denote the action map $M_n \otimes M_n^{op} \to B(M_n)$ by $\Phi$; it is actually a linear isomorphism, as verified by a simple dimension count. Let $V \subset M_n$ be a subspace. Define the corresponding left ideal $I_V \subset M_n \otimes M_n^{op}$ via $I_V:= \{x\in M_n\otimes M_n^{op}: \Phi(x)_{|V}=0\}$. It is a well-known fact that left ideals in finite dimensional $C^{\ast}$-algebras are principal and represented by a unique projection $\widetilde{P_{V}}$. The projection $P_V:= Id - \widetilde{P_{V}}$ will be the one we are interested in. So far we have established a map $V \mapsto P_V$, which is actually order preserving. It is not difficult to check that $\Phi(P_V)$ is the orthogonal projection from $M_n$ onto $V$ (with respect to the trace), so the correspondence between $V$ and $P_V$ is bijective.
\end{proof}
We will now go from the projection $P\in M_n \to M_n^{op}$ to the corresponding adjacency matrix. 
\begin{lemma}
There is a bijective correspondence between elements  $X\in M_n \otimes M_n^{op}$ and linear maps $T\in B(M_n)$ given by $T_X(x):= (Id \otimes\tau) (X(\mathds{1}\otimes x))$ and $X_{T}:= (T\otimes Id)m^{\ast}(\mathds{1})$.
\end{lemma}
\begin{proof}
Let us begin by noting that $X_T$ is essentially the Choi matrix of the map $T$: we have 
\begin{equation*}
X_T = \frac{1}{n} \sum_{i,j=1}^{n} T(e_{ij}) \otimes e_{ji}. 
\end{equation*}
On the other hand, any $X \in M_n\otimes M_n^{op}$ can uniquely represented as $X = \frac{1}{n}\sum_{i,j}^{n} T_{ij} \otimes e_{ji}$. By a simple computation we get 
\begin{equation*}
T_X(x) = \sum_{i,j=1}^{n} T_{ij} \op{Tr}(xe_{ji}) = \sum_{i,j=1}^{n} x_{ij}T_{ij}.
\end{equation*}
It follows that $T_{ij} = T_{X}(e_{ij})$, so we indeed have a bijective correspondence.
\end{proof}
We will now describe a class of linear maps on $M_n$ for which the corresponding Choi matrix is a projection.
\begin{proposition}
  The element $P \in M_n \otimes M_n^{op}$ is a projection if and only if the corresponding linear map $A:=T_{P} \in B(M_n)$ satisfies $m(A\otimes A)m^{\ast} = A$ and is completely positive.
\end{proposition}
\begin{proof}
Recall that $P$ is of the form $P= \frac{1}{n} \sum_{i,j=1}^{n} A(e_{ij}) \otimes e_{ji}$. It is a projection if and only if it is idempotent and a positive operator. Choi's theorem says that it is positive iff the corresponding linear map $A$ is completely positive. It is idempotent if it is equal to (recall that we multiply in $M_n \otimes M_n^{op})$:
\begin{equation*}
\frac{1}{n^2} \sum_{i,j,k,l=1}^{n} A(e_{ij}) A(e_{kl}) \otimes e_{lk}e_{ji}.
\end{equation*}
It follows that $j=k$ and we are left with
\begin{equation*}
\frac{1}{n^2} \sum_{i,l=1}^{n} \left(\sum_{j=1}^{n} A(e_{ij}) A(e_{jl})\right) \otimes e_{li}.
\end{equation*}
It follows that $\frac{1}{n} \sum_{j=1}^{n} A(e_{ij}) A(e_{jl}) = A(e_{il})$, which means exactly that $m (A\otimes A)m^{\ast} = A$, when applied to $e_{il}$, so we are done.
\end{proof}
\begin{definition}
We call a completely positive map $A: M_n \to M_n$ a \textbf{quantum adjacency matrix} if it satisfies
\begin{equation*}
m(A\otimes A)m^{\ast} = A.
\end{equation*}
\end{definition}

\subsubsection{Symmetry and reflexivity in different guises}
In Weaver's terminology, a quantum graph is a symmetric and reflexive quantum relation: for an operator space $V \in M_n$ these conditions mean that $V=V^{\ast}$ and $\mathds{1} \in V$, respectively. Thus a quantum graph is exactly an operator subsystem of $M_n$. We would like to check now what do these conditions mean in terms of the associated projection and the quantum adjacency matrix. 
\begin{proposition}
Let $V\subset M_n$ and let $P_V \in M_n \otimes M_n^{op}$ be the corresponding projection. Then $V=V^{\ast}$ if and only if $P_V = \sigma(P_V)$, where $\sigma$ is an automorphism of $M_n \otimes M_n^{op}$ that flips the two tensor factors. The condition $\mathds{1} \in V$ corresponds to $m(P_V) = \mathds{1}$, where we abuse the notation slightly and view the multiplication map as acting from $M_n \otimes M_n^{op}$ to $M_n$.
\end{proposition}
\begin{proof}
Recall that $\Phi(P_V): M_n \to M_n$ is the orthogonal projection onto $V$. Therefore $\mathds{1}\in V$ iff $\Phi(P_V) (\mathds{1}) = \mathds{1}$, which can be translated into $m(P_V) = \mathds{1}$. If $P$ is the orthogonal projection onto $V$, then the orthogonal projection onto $V^{\ast}$ is given by $\widetilde{P}(x) := (Px^{\ast})^{\ast}$. If $P_V = \sum_{i} X_i \otimes Y_i$ then $\Phi(P_V)(x) = \sum_{i} X_i x Y_i$. It follows that the orthogonal projection onto $V^{\ast}$ is equal to $x \mapsto\sum_{i} Y_{i}^{\ast} x X_i^{\ast}$. Recall that $P_V=\sum_{i} X_{i}\otimes Y_i= \sum_{i} X_{i}^{\ast}\otimes Y_i^{\ast}$ is self-adjoint, so the projection onto $V^{\ast}$ corresponds to $\sigma(P_V)$, i.e. $V=V^{\ast}$ iff $P_V = \sigma(P_V)$.
\end{proof}
We will now check what the symmetry and reflexivity conditions mean in terms of the quantum adjacency matrix.
\begin{proposition}\label{Prop:symrefl}
Let $V\subset M_n$. We have $V=V^{\ast}$ if and only if the quantum adjacency matrix $A: M_n \to M_n$ is self-adjoint with respect to the Hilbert-Schmidt inner product. The condition $\mathds{1} \in V$ is equivalent to the equality $m(A\otimes Id)m^{\ast}(\mathds{1}) = \mathds{1}$.
\end{proposition}
\begin{proof}
Let $P_V$ be the projection corresponding to $V$. Suppose that $V=V^{\ast}$ and hence $P_V = \sum_{i} P_i\otimes Q_i = \sum_{i} Q_i\otimes P_i$. The adjacency matrix is equal to $Ax = \sum_{i} P_i \tau(Q_i x) = \sum_{i} Q_i \tau(P_i x)$. If we compute 
\begin{equation*}
\tau((Ax) y) = \sum_{i} \tau(P_i y) \tau(Q_i x) = \sum_{i} \tau(Q_i y) \tau(P_i x),
\end{equation*}
we quickly see that $V=V^{\ast}$ is equivalent to $\tau((Ax) y) = \tau(x (Ay))$, which combined with complete positivity of $A$ shows that it is self-adjoint, as completely positive maps are $\ast$-preserving.

Consider now the condition $\mathds{1}\in V$. Recall that $P_V = \frac{1}{n} \sum_{i,j=1}^{n} A(e_{ij}) \otimes e_{ji}$, so $m(P_V)=\mathds{1}$ means that
\begin{equation*}
\frac{1}{n}\sum_{i,j} A(e_{ij}) e_{ji} = \mathds{1}.
\end{equation*}
The left-hand side is equal to $m(A\otimes Id)m^{\ast}(\mathds{1})$.
\end{proof}

\subsection{Correspondence between quantum relations and quantum adjacency matrices made concrete}

It will be useful to have a number of different perspectives on quantum graphs. For example, it is very difficult to come up with a good random model of a quantum adjacency matrix, but it is very simple, when we look at the corresponding subspace of a matrix algebra. In order to benefit from the situation, we need a more concrete way of translating from one approach to the other. This is the aim of this subsection.

\begin{proposition}\label{Prop:quantumadj}
  Let $V\subset M_n$ be a $d$-dimensional subspace with an orthonormal basis (with respect to the usual trace) $(A_i)_{i\in [d]}$. Then the orthogonal projection $P_V \in M_n \otimes M_n^{op}$ is equal to $\sum_{i,k,l} A_i e_{kl} \otimes A_{i}^{\ast} e_{lk}$. The quantum adjacency matrix is given by $Ax = n\sum_{i=1}^{d} A_i x A_i^{\ast}$.
\end{proposition}
\begin{proof}
The orthogonal projection onto $V$ is given by $Px = \sum_{i=1}^{d} A_i \op{Tr}(x A_i^{\ast})$. Note that $\op{Tr}y \mathds{1} = \sum_{k,l} e_{kl} y e_{lk}$, so 
\begin{equation*}
Px = \sum_{i,k,l} A_i e_{kl} x A_i^{\ast} e_{lk}.
\end{equation*}
It follows that for $X:= \sum_{i,k,l} A_i e_{kl} \otimes A_i^{\ast} e_{lk}$, $\Phi(X) = P$, so $X  =P_V$. If we have our projection, we can easily compute the adjacency matrix
\begin{equation*}
Ax = \sum_{i,k,l} A_i e_{kl} \tau(x A_{i}^{\ast} e_{lk}) = n\sum_{i,k,l} A_i e_{kl} (xA_{i}^{\ast})_{kl}.
\end{equation*}
Note that $A_{i} e_{kl} = \sum_{r} (A_{i})_{rk} e_{rl}$, so we obtain
\begin{equation*}
Ax = n\sum_{i,k,l,r} (A_{i})_{rk} (xA_{i}^{\ast})_{kl} e_{rl}. 
\end{equation*}
After summing over $k$ we finally get
\begin{equation*}
Ax = n\sum_{i} A_i x A_i^{\ast}.
\end{equation*}
\end{proof}
\subsubsection{Automorphisms of quantum graphs}
We will now see how to define automorphisms of quantum graphs in all three pictures described above.
\begin{definition}
Let $V \subset M_n$ be a quantum relation. An automorphism of $V$ is a $\ast$-automorphism $\Phi: M_n \to M_n$ such that $\Phi(V) = V$. Since any $\ast$-automorphism of $M_n$ is given by a conjugation by a unitary, we will usually identify automorphisms of quantum relations with unitaries $U \in M_n$ such that $UVU^{\ast} = V$.
\end{definition}
\begin{proposition}\label{pr:vinv}
Let $V \subset M_n$ be a quantum relation. Let $P = \sum_{i} P_{i} \otimes Q_{i} \in M_n \otimes M_n^{\op{op}}$ be the corresponding orthogonal projection and let $A:M_n \to M_n$ be the quantum adjacency matrix. Then a unitary $U\in M_n$ is an automorphism of $V$ iff $P = \sum_{i} UP_{i} U^{\ast} \otimes UQ_{i} U^{\ast}$ iff $U^{\ast} A(x) U = A(U^{\ast}xU)$. In particular, any unitary $U$ that gives an automorphism commutes with the \textbf{degree matrix} $D:= A\mathds{1}$.
\end{proposition}
\begin{proof}
Let $P_V: M_n \to M_n$ be the orthogonal projection onto $V$, represented by $P \in M_n \otimes M_n^{\op{op}}$. The orthogonal projection onto $UVU^{\ast}$ is given by $P_{UVU^{\ast}}x = U P_V(U^{\ast}x U) U^{\ast}$. It readily follows that if $P =\sum_{i} P_{i} \otimes Q_{i}$ then $P_{U} \in M_n \otimes M_n^{op}$ corresponding to $P_{UVU^{\ast}}$ is equal to $\sum_{i} UP_{i} U^{\ast} \otimes UQ_{i} U^{\ast}$. 

Consider now the quantum adjacency matrix (see the proof of \Cref{Prop:symrefl}) $Ax:= \sum_{i} P_i \tau(Q_i x)$. The quantum adjacency matrix corresponding to $P_{U}$ is given by
\begin{equation*}
  A_{U}x:= \sum_{i} U P_{i} U^{\ast} \tau(x U Q_i U^{\ast}) = \sum_{i} U P_{i} U^{\ast} \tau (U^{\ast} x U Q_i).
\end{equation*}
The two are equal if and only if $A(U^{\ast} x U) = U^{\ast} A(x) U$. If we apply this equality to $x=\mathds{1}$ then we get $A\mathds{1} = U^{\ast} A\mathds{1} U$, that is $U D = D U$.
\end{proof}

A tensor-category-theoretic interpretation of \Cref{pr:vinv} would run as follows: the compact group $U_n$ acts on $M_n$ by conjugation, and consider the subgroup $G\le U_n$ leaving $V$ invariant. The dual of $M_n$ as a $U_n$-representation is isomorphic to the complex conjugate space $\overline{M_n}\cong M_n$, via the trace pairing:
\begin{equation*}
  M_n\otimes M_n\ni A\otimes B\mapsto \mathrm{Tr}(A^*B). 
\end{equation*}
Under $G$ the dual $\overline{V}\subset \overline{M_n}$ is also conjugation-invariant, and upon identifying
\begin{equation*}
  V\otimes \overline{V} \cong \mathrm{End}(V)
\end{equation*}
the identity $\id\in \mathrm{End}(V)$, corresponding to
\begin{equation}\label{eq:db}
  \sum_{i=1}^e S_i\otimes T_i\in V\otimes \overline{V}
\end{equation}
for trace-dual bases $\{S_i\}\subset V$ and $\{T_i\}\subset \overline{V}$, must also be $G$-invariant. In short: conjugation by $U\otimes U$, $U\in G$ fixes \Cref{eq:db} for any
\begin{itemize}
\item basis $\{S_i\}$ for $V$;
\item with dual basis $\{T_i\}$ with respect to the trace pairing. 
\end{itemize}

Note furthermore that since $M_n\cong \overline{M_n}$ via the adjoint map $A\mapsto A^*$ and $V$ is self-adjoint (being an operator system), we have $\overline{V}=V$. Finally, the degree matrix is nothing but the image of \Cref{eq:db} through the multiplication $M_n\otimes M_n\to M_n$, which itself is $U_n$-equivariant.

\subsection{Random models of quantum graphs}
In this section we will describe the random models that we will be dealing with. For classical graphs there are two main models of interest (both called Erd\H{o}s-R\'{e}nyi models): $G(n,p)$, where $0<p<1$ and $G(n,M)$, where $0\leqslant M \leqslant \binom{n}{2}$. In the first one two vertices are connected by an edge with probability $p$, independently over all pairs of vertices, and $G(n,M)$ is a uniform distribution over all graphs with exactly $M$ edges. In the $G(n,p)$ the probability that the graph has exactly $M$ edges is given by the binomial distribution
\begin{equation*}
  \binom{\binom{n}{2}}{M}p^{M} (1-p)^{\binom{n}{2} - M}
\end{equation*}
and all graphs with $M$ edges have exactly the same probability. It follows that the $G(n,p)$ can be recovered as a weighted sum of the $G(n,M)$ models and this is the approach that we will take, since there is no reasonable notion of a vertex of a quantum graph, so it does not seem feasible to define the $G(n,p)$ model directly.

Traditionally random graphs are undirected and we will follow this tradition. Moreover we will consider reflexive, undirected graphs, i.e. we will be dealing with operator subsystems of the matrix algebra $M_n$. The dimension of such a subsystem can be interpreted as the average degree of our quantum graph; indeed, by \Cref{Prop:quantumadj} the normalized trace of the degree matrix $D=A\mathds{1}$ is equal to $d$, the dimension of our operator system. Actually, since we are dealing with reflexive graphs, it is more reasonable to call the number $d-1$ the average degree.

In order to find a good random model of operator subsystems of a given dimension $d+1$, we can notice that that its sufficient to take a random $d$-dimensional subspace of Hermitian matrices and add identity to it. Among the most famous models of Hermitian matrices is the GUE (the Gaussian unitary ensemble), which we will employ. Recall that in this model the probability measure has a density with respect to the Lebesgue measure on Hermitian matrices proportional to $\exp(-\frac{n}{2} \op{Tr} (A^2))$ (see \cite[\S 2.6]{MR2906465}).
\begin{definition}
Fix $n\in \bN$ and $0\leqslant d \leqslant n^2-1$. We define the model $QG(n,d)$ to be a random operator system $V_d:= \op{span}(\mathds{1}, X_{1}, \dots, X_d\}$, where $X_1,\dots, X_d$ are independent GUE $n\times n$ matrices. Actually, we will take $X_{i}$ to be the traceless GUE matrix, i.e. the centered version of GUE which is constructed from the usual GUE by subtracting an appropriate multiple of the identity.
\end{definition}

\begin{remark}
  In the specified range of $d$ the GUE matrices $X_1,\dots, X_d$ are almost surely linearly independent, so their span is a $d$-dimensional subspace of the space $M_{n,H}^{0}$ of traceless Hermitian matrices, which has (real) dimension $n^2-1$. We have therefore a map from a full measure subset of $\left(M_{n,H}^{0}\right)^{d}$ to the Grassmannian $\op{Gr}(d,n^2-1)$, which maps the tuple $(X_1,\dots, X_d)$ to its span. We can consider the probability measure on the Grassmannian by pushing forward via this map. As the density of GUE is proportional to $\exp(-\op{Tr} X^2)$, it is invariant under all maps on $M_{n,H}^{0}$ orthogonal with respect to the Hilbert-Schmidt product, hence so is its pushforward. It follows that the measure on the Grassmannian is invariant under $O(n^2-1)$, i.e. it has to be the unique Haar measure. This is the right analogue of a uniform distribution appearing in the classical $G(n,M)$ model.
\end{remark}
It is not difficult now to check that \Cref{Cor:trivautopsys} is a consequence of \Cref{Thm:trivaut}. Before doing that, recall that Hamana introduced in \cite{MR566081} the $C^{\ast}$-envelope of an operator system $S$, which is the unique $C^{\ast}$-algebra $C^{\ast}(S)$ containing $S$ completely order isomorphically, such that whenever $S$ is embedded completely order isomorphically inside a $C^{\ast}$-algebra $B$ then the identity on $S$ extends to a quotient from the $C^{\ast}$-algebra generated by $S$ inside of $B$ onto $C^{\ast}(S)$. Recall also that an automorphism of an operator system $S$ is a linear bijection $\Phi: S \to S$ that is a unital complete order isomorphism, i.e. both $\Phi$ and $\Phi^{-1}$ are completely positive.
\begin{proposition}\label{Prop:opsysaut}
Fix $n\geqslant 3$ and $2\leqslant d \leqslant n^2-3$. If the automorphism group of the random quantum graph $QG(n,d)$ is almost surely trivial, then the automorphism group of the underlying operator system is almost surely trivial.
\end{proposition}
\begin{proof}
Since $d\geqslant 2$, the $C^{\ast}$-algebra generated by $V_{d}$ is almost surely equal to $M_n$, as two independent GUE matrices almost surely generate the full matrix algebra (see \Cref{lem:twogenerate}). Since the algebra $M_n$ is simple, it follows that the $C^{\ast}$-envelope of $V_d$ is equal to $M_n$. It follows that any unital complete order automorphism of $V_d$ gives rise to an automorphism of $M_n$, which is exactly an automorphism of $V_d$ viewed as a quantum graph. It follows that the automorphism of $V_d$ had to be the identity.
\end{proof}

It is time to introduce the analogue of the $G(n,p)$ model for quantum graphs.
\begin{definition}
Fix $n\in \bN$ and $0<p<1$. We define the random model $QG(n,p)$ in the following way: we first choose a random $0\leqslant d \leqslant n^2-1$ according to the binomial distribution $B(n^2-1,p)$ and then choose a random operator subsystem of $M_n$ according to the model $QG(n,d)$.
\end{definition}
Here we can already show that \Cref{Cor:trivaut} follows easily from \Cref{Thm:trivaut}.
\begin{proposition}\label{Prop:qgnp}
If for any $d$ between $2$ and $n^2-3$ the automorphism group of a quantum graph in the $QG(n,d)$ model is almost surely trivial then the automorphism group of a quantum graph in the $QG(n,p(n))$ model is asymptotically almost surely trivial as soon as $\lim_{n\to\infty} n^2 p(n)=\infty$ and $\lim_{n\to\infty} n^2(1-p(n))=\infty$. If $p(n)$ satisfies $\lim_{n\to\infty} n^{2-\varepsilon} p(n)=\infty$ and $\lim_{n\to\infty} n^{2-\varepsilon}(1-p(n))=\infty$ for some $\varepsilon > 0$ then it holds with overwhelming probability.
\end{proposition}
\begin{proof}
Assume that for $d$ between $2$ and $n^2-3$ the automorphism group is almost surely trivial. Note that it can be nontrivial if $d$ happens to be equal to $0$, $1$, $n^2-2$, or $n^2-1$. In the model $QG(n,p)$ number $d$ is chosen according to the binomial distribution $N(n^2-1,p)$ (i.e. $0\leqslant d \leqslant n^2-1$), so the probability that it is equal to one of those four numbers is equal to $P_n:= (1-p)^{n^2-1} + (n^2-1)p(1-p)^{n^2-2} + (n^2-1)p^{n^2-2}(1-p) + p^{n^2-1}$. Let us look at the first term $(1-p)^{n^2-1}$. It can be rewritten as $\left((1-p)^{\frac{1}{p}}\right)^{(n^2-1)p}$. By the inequality $1-x \leqslant e^{-x}$, we have that $(1-p)^{\frac{1}{p}} \leqslant e^{-1}$, so this whole expression is bounded above by $\exp(-(n^2-1)p))$. If $p(n)\cdot n^2$ converges to infinity, then this expression converges to $0$. If we denote $q_n:= (n^2-1)p(n)$, then the second term in in $P_n$ can be bounded above by $q_n e^{-q_n}$, so it also converges to $0$. The last two terms can be handled in the same way provided that we have $\lim_{n\to\infty} (1-p(n))n^2=\infty$. 

If the slightly stronger assumptions $\lim_{n\to\infty} n^{2-\varepsilon} p(n)=\infty$ and $\lim_{n\to\infty} n^{2-\varepsilon}(1-p(n))=\infty$ hold, then one can take advantage of the exponential decay to show that the automorphism group is trivial with overwhelming probability.
\end{proof}
\section{Small quantum automorphism groups}\label{se:qaut}

The present section is devoted to results to the effect that various structures (graphs, quantum sets) have small quantum automorphism groups ``generically''.

\subsection{Enhanced quantum sets}

\begin{definition}\label{def.qs}
  A {\it finite measured quantum set} is a pair $X=(\cO(X),\psi_X)$ consisting of a finite-dimensional $C^*$-algebra $\cO(X)$ and a state $\psi_X$ on it. The {\it size} or {\it cardinality} $|X|$ of $X$ is $\dim \cO(X)$. 

  We often abbreviate the term to `quantum set' and write $L^2(X)$ for the Hilbert space structure induced on $\cO(X)$ by the state.

  An {\it enhanced quantum set} is a finite measured quantum set $X$ equipped with a self-adjoint operator $A_X$ on $L^2(X)$. An enhanced quantum set is {\it demi-classical} if the underlying algebra $\cO(X)$ is commutative.

  Unless specified otherwise states $\psi_X$ as above are assumed faithful.
\end{definition} 
The quantum automorphism groups of finite dimensional $C^{\ast}$-algebras have been defined in \cite{Wan98}, while the quantum automorphism groups have been defined  in \cite{MR2146039}. The definition of a quantum automorphism group of an enhanced quantum set can be naturally deduced from these two cases. We will only consider the quantum automorphism group of demi-classical enhanced quantum sets, so we provide a precise definition in this context.
\begin{definition}\label{def:qaut}
Let $[n]:=\{1,\ldots, n\}$ be a finite set, $\psi: \bC^n \to \bC$ a state and $A: \bC^n \to \bC^n$ a self-adjoint map with respect to the inner product induced by $\psi$. The quantum automorphism group of the finite set $[n]$ is the universal $C^{\ast}$-algebra $C(S_n^{+})$ generated by elements $(u_{ij})_{i,j\in [n]}$ subject to relations
\begin{itemize}
\item the matrix $U:=(u_{ij})_{i,j\in [n]}$ is unitary;
\item each $u_{ij}$ is a projection, i.e. $u_{ij} = u_{ij}^{\ast} = u_{ij}^2$;
\item rows and columns of the matrix $U$ sum to the identity, i.e. $\sum_{i=1}^{n} u_{ij} = \mathds{1}$ and $\sum_{i=1}^{n} u_{ji} = \mathds{1}$ for all $j\in [n]$.
\end{itemize}
The assignment $u_{ij} \mapsto \sum_{k=1}^{n} u_{ik} \otimes u_{kj}$ extends to a coassociative comultiplication $\Delta: C(S_n^{+}) \to C(S_n^{+})\otimes C(S_n^{+})$, endowing $C(S_n^{+})$ with the structure of a compact quantum group.

Let $\alpha: \bC^n \to \bC^n \otimes C(S_n^{+})$ be the universal action of $S_n^{+}$ on $\bC^n$, i.e. the map $e_{i}\mapsto \sum_{j=1}^{n} e_{j} \otimes u_{ji}$. The quantum automorphism group of the demi-classical enhanced quantum set $(\bC^n, \psi, A)$ is the quotient  of $C(S_n^{+})$ by the following relations:
\begin{itemize}
\item $\alpha \circ A = (A\otimes \op{Id}) \circ \alpha$;
\item $(\psi \otimes \op{Id})\circ \alpha(\cdot) = \psi(\cdot)\mathds{1}$.
\end{itemize}
\end{definition}
We give sufficient conditions that will ensure the quantum automorphism group of a demi-classical enhanced quantum set is trivial. This requires some preparation.

\begin{definition}\label{def.overlap}
  Let $(e_i)_{i=1}^n$ be a basis of a finite-dimensional vector space $V$. Two vectors $v,w\in V$ {\it overlap with respect to $(e_i)$} (or simply {\it overlap} when the basis is understood) if for at least one of the $e_i$'s the coefficients of both $v$ and $w$ with respect to this $e_i$ are non-zero.

  An operator $A\in \mathrm{End}(V)$ is {\it thick (with respect to $(e_i)$)} if all of its eigenvectors mutually overlap.
\end{definition}

In the context of a demi-classical enhanced quantum set the basis will always be the standard one, consisting of minimal projections of $\cO(X)$, and we will be typically interested in the thickness of the adjacency matrix $A_X$.

\begin{proposition}\label{pr.trv}
  Let $X=(\cO(X),\psi_X,A_X)$ be a demi-classical enhanced quantum set whose adjacency matrix $A_X$ has simple spectrum and is thick. Then, the quantum automorphism group $G_X$ is isomorphic to $(\bZ/2)^n$ for some $n$.
\end{proposition}
\begin{proof}
  It is enough to argue that $G_X$ is classical, for then the fact that it is a $2$-group follows:
  
  Suppose we know that $G_X$ is classical. Then every $T\in G_X$ operates as a permutation on the classical set $X$ and scales each eigenvector $v$ of $A_X$ (because the spectrum is simple): $Tv=\lambda v$ for some $\lambda\in \mathbb{C}$. Since $T$ permutes the components of $v$ the scalar $\lambda$ must be a root of unity, and the fact that $v$ can be chosen to be a {\it real} vector implies that $\lambda=\pm 1$.

  It thus remains to argue that under the hypotheses $G_X$ is indeed classical. Denote by $\alpha$ the coaction of $G_{X}$. Recall (see \Cref{def:qaut}) that $\alpha \circ A_{X} = (A_{X}\otimes \op{Id})\circ \alpha$. If $v$ is an eigenvector of $A_{X}$, i.e. $A_{X}v=cv$ then we get $c\alpha(v) = (A_{X}\otimes \op{Id})\alpha(v)$. It follows that $\alpha(v)$ is an eigenvector of $(A_{X}\otimes \op{Id})$ with eigenvalue $c$, so 
  \begin{equation}\label{eq:1}
   \alpha(v) = v\otimes g_v,
  \end{equation}
because the spectrum of $A_X$ is simple. Moreover each $g_v$ is group-like, i.e. $\Delta(g_v) = g_v \otimes g_v$. It follows that $\cO(G_X)$ is the group algebra of a finite group $\Gamma$ generated by the $g_v$ (as $v$ ranges over an eigenbasis for $A_X$), so it will be enough to show that every two $g_v$ commute (for then $\Gamma$ is abelian and $G_X=\widehat{\Gamma}$ will be its Pontryagin dual).

  So far only the vector space structure of $\cO(X)$ was important, but to finish the proof we need to recall that it is also a \emph{commutative} algebra. It follows from the formula \Cref{eq:1} for the coaction that $\alpha(vw) = vw \otimes g_v g_w$ and $\alpha(wv) = wv \otimes g_w g_v$ and the two are equal, as $wv=vw$. Therefore we get $vw\otimes g_v g_w = vw \otimes g_w g_v$. The thickness hypothesis shows that $vw\neq 0$ and hence $g_w g_v = g_v g_w$. This finishes the proof. 
\end{proof}

\subsection{Plain graphs}

Here we give an alternative proof of \cite[Theorem 3.14]{MR4097284}, stating that as the size $n$ goes to infinity, the probability that the random graph $G\left(n,\frac 12\right)$ has trivial quantum automorphism group goes to $1$. We also generalize that result slightly in that the random graphs being considered are $G(n,p)$ rather than $G\left(n,\frac 12\right)$.

\begin{theorem}\label{pr.grph-rnd}
  Let $p\in (0,1)$. As $n\to\infty$ the probability that the random graph $G(n,p)$ has trivial quantum automorphism group approaches $1$.
\end{theorem}
\begin{proof}
  We know from \cite[Theorem 1.3]{tv-smpl} or \cite[Theorem 1.4]{ovw} that overwhelmingly (see \Cref{def.hp}), the adjacency matrix $A=A_X$ of a random graph has simple spectrum. Furthermore, the second part of \cite[Theorem 1.4]{ovw} also shows that overwhelmingly, no $A$-eigenvectors have any vanishing coordinates.

  This latter condition (all-non-zero coordinates) is stronger than thickness in the sense of \Cref{def.overlap}, and hence we conclude from \Cref{pr.trv} that with overwhelming probability for $X=G(n,p)$, the quantum automorphism group $G_X$ is of the form $(\bZ/2)^d$ and in particular classical.

  We can now finish using the well known fact that fact almost all random graphs $G(n,p)$ have trivial {\it classical} automorphism group (e.g. \cite[Theorem 2]{er} for the case $p=\frac 12$ and \cite[Theorem 3.1]{ksv} in general).
\end{proof}

\subsection{Regular and isospectral graphs}

The generic rigidity result in \Cref{pr.grph-rnd} also applies to random regular graphs $G(n,r)$.

\begin{theorem}\label{pr.reg}
Fix a positive integer $r\ge 3$. As $n\to \infty$ the probability that the random regular graph $G(n,r)$ has trivial quantum automorphism group approaches $1$. 
\end{theorem}
\begin{proof}
  Let $x$ be a vertex of a finite graph. Following \cite{bol-dist} we denote by $d_i(x)$ the number of vertices at distance $i$ from $x$. According to \cite[Theorem 6]{bol-dist} the probability that the vertices of a random regular graph have distinct sequences
  \begin{equation*}
    d_1(x),\ d_2(x),\ \cdots
  \end{equation*}
  approaches $1$. Since two vertices $x,y$ in the same orbit of a quantum action on the graph will have
  \begin{equation*}
    d_i(x)=d_i(y),\ \forall i\ge 1,
  \end{equation*}
  the conclusion follows.
\end{proof}

\Cref{pr.reg} will allow us to extend \cite[Example 4.18]{MR4091496}, obtaining an abundance of pairs of non-quantum-isomorphic, isospectral graphs with trivial quantum automorphism groups. The construction mimics that in loc. cit., based around the procedure outlined in \cite{gm} for producing isospectral graphs. Recall \cite[Theorem 2.2]{gm}:

\begin{theorem}\label{th.gm}
  Let $\Gamma$ be a regular graph with $2m$ vertices and partition its vertex set as $V=V'\sqcup V''$ with
  \begin{equation*}
    |V'|=|V''|=m.
  \end{equation*}
  The two graphs $\Gamma'$, $\Gamma''$ (with $2m+1$ vertices) obtained by adding a vertex to $V$ and connecting it to $V'$ (respectively $V''$) are isospectral.\qedhere
\end{theorem}

The following result proves the existence of the announced examples. 

\begin{proposition}\label{pr.g'g''}
  Fix a positive integer $r\ge 3$ and consider a random regular graph $\Gamma=G(2m,r)$ as $m\to\infty$, with $\Gamma'$ and $\Gamma''$ as in \Cref{th.gm}. 
  \begin{enumerate}[(a)]
  \item\label{item:1}   As $m\to \infty$ the probability that $\Gamma'$ and $\Gamma''$ have trivial quantum automorphism groups approaches $1$.
  \item\label{item:2} The same holds for the probability that $\Gamma'$ and $\Gamma''$ are not quantum isomorphic. 
  \end{enumerate}
\end{proposition}
\begin{proof}
  We focus on \labelcref{item:1}, as the proof of \labelcref{item:2} is parallel.

  For $m>r$ $\Gamma'$ has a single vertex of degree $>r$. Since vertices in the same orbit of a quantum action on a graph have equal degrees, any such action on $\Gamma'$ will preserve both the extra vertex and the subgraph $\Gamma\subset \Gamma'$. But by \Cref{pr.reg} the probability that the latter action on $\Gamma$ is trivial approaches $1$ as $m\to\infty$, hence the conclusion.
  
 To prove \labelcref{item:2}, note that a quantum isomorphism between $\Gamma'$ and $\Gamma''$ fixes the vertex added to $\Gamma$, so the magic unitary implementing it has a block form, where the block corresponding to the subgraph $\Gamma$ provides a nontrivial quantum automorphism; by \Cref{pr.reg} this can happen with probability tending to $0$ as $m\to \infty$.
\end{proof}

\section{Generic small automorphism groups}\label{se:gen}
In this section we study automorphism groups of quantum graphs.
\subsection{Conjugation actions on tuples}\label{subse:cjact}

We need a number of remarks on the conjugation action by $GL_n$, $U_n$ and/or their projective versions $PSL_n$, $PSU_n$, etc. on spaces spanned by tuples of matrices. The results have an algebraic-geometric flavor, revolving around the geometric invariant theory of reductive-group actions; we collect those results in the present subsection for future reuse. We assume some familiarity with linear algebraic groups as covered in a number of good sources: \cite{hum,bor}, etc.

First, we fix some conventions.

\begin{convention}\label{cv:perm}
  Throughout,
  \begin{equation}\label{eq:fam}
  \cX:=\{X_i,\ 1\le i\le d\}
\end{equation}
denotes a $d$-tuple of $n\times n$ matrices, with $d$ and $n$ fixed ahead of time and understood from context; we occasionally vary the base symbol `$X$' as in $\cY$ denoting a tuple of $Y_i$, etc. We say that $\cX$ is {\it self-adjoint} if $X_i$ are, i.e. $X_i^*=X_i$.

With that in place,
\begin{equation}\label{eq:spanxi}
  \langle X_i\rangle:=\mathrm{span}\{X_i,\ 1\le i\le d\}
\end{equation}
will be the span of the matrices.

The tuple \Cref{eq:fam} is
\begin{itemize}
\item {\it linearly independent}, or
\item {\it self-adjoint}, or
\item {\it traceless} 
\end{itemize}
if all the $X_i$ are. 

For linearly independent $\cX$ we write
\begin{equation*}
  \Lambda\cX:= X_1\wedge\cdots\wedge X_d\in \Lambda^d M_n
\end{equation*}
for the corresponding element of the exterior power of the space of matrices. The span $\langle \cX\rangle$ can then be identified with an element of the Grassmannian $\mathrm{Gr}(d,M_n)$ of $d$-dimensional subspaces of $M_n$, and that element is identified with the line $\bC \Lambda\cX\in \bP(\Lambda^d M_n)$ (regarded as an element of the projective space of $\Lambda^d M_n$) upon embedding
\begin{equation*}
  \mathrm{Gr}(d,M_n)\subset \bP(\Lambda^d M_n)
\end{equation*}
via the Pl\"ucker embedding \cite[Example 6.6]{har}.
\end{convention}

The subsequent discussion will hinge crucially on the following observation.

\begin{lemma}\label{le:clorb}
  Let $\cX$ be a linearly independent self-adjoint tuple \Cref{eq:fam}. Then, the orbit of $\Lambda \cX\in \Lambda^d M_n$ under the conjugation action by $G:=PSL_n$ is Zariski-closed.
\end{lemma}
\begin{proof}
  The complex algebraic variety $G$ has a {\it real structure} \cite[AG \S 11]{bor} whose underlying real points make up precisely the maximal compact subgroup $G_{\bR}:=PSU_n$ of $G$. Note moreover that conjugation by unitary operators preserves the real vector space
  \begin{equation*}
    M_{n,sa}:=\{\text{self-adjoint matrices}\}
  \end{equation*}
  which in turn is a real structure for $M_n$ in the sense that
  \begin{equation*}
    M_n\cong M_{n,sa}\otimes_{\bR}\bC. 
  \end{equation*}
  It follows that $\Lambda^d M_n$ is a representation of $G$ defined over $\bR$. Since the orbit of $\Lambda\cX$ is closed (in the standard topology) under the compact group $G_{\bR}$, the conclusion follows from \cite[Corollary 5.3]{birk}.
\end{proof}

In particular, we also have

\begin{corollary}\label{le:fixisred}
  Let $\cX$ be a linearly independent self-adjoint tuple \Cref{eq:fam}. Then, the stabilizer group $G_{\Lambda \cX}$ of $\Lambda\cX\in \Lambda^d M_n$ under the conjugation action by $G:=PSL_n$ is reductive.
\end{corollary}
\begin{proof}
  Indeed, by \cite[\S I.2, Proposition]{luna} (attributed there to Matsushima \cite{mats}) this is a consequence of the closure of the orbit
  \begin{equation*}
    G(\Lambda \cX)\subset \Lambda^d M_n,
  \end{equation*}
  which in turn is asserted by \Cref{le:clorb}.
\end{proof}

\begin{theorem}\label{th:gencj}
  Let $\cX$ be a linearly independent self-adjoint tuple \Cref{eq:fam} and set $G:=PSL_n$ and $K:=PSU_n$. Then:
  \begin{enumerate}[(a)]
  \item\label{item:3} For $\Lambda\cY\in \Lambda^d M_n$ in a Zariski-open neighborhood of $\Lambda\cX$ the isotropy group $G_{\Lambda\cY}$ is conjugate to a subgroup of $G_{\Lambda\cX}$.
  \item\label{item:4} The same goes for the isotropy groups $K_{\bullet}$ with $\bullet\in\{\Lambda\cY,\Lambda\cX\}$ and self-adjoint $\cY$.
  \item\label{item:5} Finally, the analogous result holds for the isotropy groups $K_{\bullet}$, $\bullet\in\{\langle\cY\rangle,\langle\cX\rangle\}$ and self-adjoint $\cY$ again, this time the action being that of $K$ on the Grassmannian $\mathrm{Gr}(d,M_n)$.
  \end{enumerate}
\end{theorem}
\begin{proof}
  We prove the three claims separately.

  {\bf \Cref{item:3}:} As observed in \cite[Remark (b) following Proposition 3.3]{rich-zero}, this follows from
  \begin{itemize}
  \item the Zariski-closure of the orbit $G(\Lambda \cX)$ (provided by \Cref{le:clorb}) and 
  \item Luna's \'etale slice theorem \cite[\S III.1]{luna}. 
  \end{itemize}

  {\bf \Cref{item:4}:} For self-adjoint $\cY$ the isotropy group $G_{\Lambda\cY}$ is invariant under the complex-conjugate-linear involution $\Theta:x\mapsto (x^*)^{-1}$ and
  \begin{equation}\label{eq:kbul*}
    K_{\bullet} = \{g\in G_{\bullet}\ |\ \Theta(x)=x\}. 
  \end{equation}
  This then implies
  \begin{equation}\label{eq:gtok}
    G_{\Lambda\cY}\subseteq G_{\Lambda\cX} \Rightarrow K_{\Lambda\cY}\subseteq K_{\Lambda\cX}. 
  \end{equation}

  Now note furthermore that the inner product
  \begin{equation*}
    \langle A,B\rangle:=\mathrm{tr}(A^*B),\ A,B\in M_n
  \end{equation*}
  induces one on $\Lambda^d M_n$ making the latter into a Hilbert space in such a manner that the conjugation action by $G$ respects the $*$-structures: if we denote that representation by $\rho:G\to GL(\Lambda^d M_n)$  then
  \begin{equation*}
    \rho(X^*) = \rho(X)^*\in GL(\Lambda^d M_n). 
  \end{equation*}
  Coupled with \Cref{eq:kbul*} this means that once we identify $G_{\bullet}$ with a subgroup of $GL(\Lambda^d M_n)$, $K_{\bullet}$ is its intersection with the unitary group of $U(\Lambda^d M_n)$. $G_{\bullet}$ thus has a polar decomposition with respect to $K_{\bullet}$ respectively (e.g. \cite[the result labeled `Mostow's Theorem', p.7]{brgr}) and the latter are maximal compact subgroups of the (reductive, by \Cref{le:fixisred}) groups $G_{\bullet}$. It follows that
  \begin{equation*}
    K_{\Lambda\cY}\subseteq K_{\Lambda\cX} \Rightarrow G_{\Lambda\cY}\subseteq G_{\Lambda\cX}. 
  \end{equation*}
  by taking Zariski closures, supplementing \Cref{eq:kbul*} and finishing the proof of \Cref{item:4}. 
  
  {\bf \Cref{item:5}:} Preserving the span $\langle \cY\rangle\in \mathrm{Gr}(r,M_n)$ means preserving the {\it line} through $\Lambda \cY\in \Lambda^d M_n$. The group $K$ acts by unitary transformations on the {\it real} Hilbert space
  \begin{equation*}
    \Lambda^d_{\bR} M_{n,sa} \subset \Lambda^d M_{n}\cong \Lambda^d_{\bR} M_{n,sa}\otimes_{\bR} \bC,
  \end{equation*}
  where $M_{n,sa}$ denotes self-adjoint matrices, as in the proof of \Cref{le:clorb}, and `$\Lambda^d_{\bR}$' denotes exterior powers of real vector spaces over $\bR$. It follows that under the $K$-action preserving the line means acting on it as $\pm 1$. But then, for elements of $K$, fixing $\langle \cY\rangle$ is equivalent to fixing
  \begin{equation*}
    \Lambda \cY\otimes \Lambda\cY\in (\Lambda^d M_n)^{\otimes 2}.
  \end{equation*}
  The same argument used in the proof of \Cref{item:4} now applies with $\Lambda\cX^{\otimes 2}$ and $(\Lambda^d M_n)^{\otimes 2}$ in place of $\Lambda\cX$ and $\Lambda^d M_n$ respectively (and similarly for $\cY$).
\end{proof}

Note also the following consequence regarding principal orbits for the actions of $PSL_n$ and $PSU_n$ on the self-adjoint Grassmannians. Recall (e.g. \cite[Introduction]{rich-zero} or \cite[discussion following Proposition I.2.5]{aud}) that a {\it principal orbit} for an action is an orbit $Gx$ with the property that the isotropy groups $G_y$ are conjugate to $G_x$ throughout a dense open set $U$ of points $y$.

In the present context we are interested in having $U$ dense open in the {\it Zariski} (rather than usual) topology, hence the phrase {\it Zariski-principal} in the following statement. 

\begin{corollary}\label{cor:princorb}
  The actions of
  \begin{itemize}
  \item $G:=PSL_n$ on linearly-independent, self-adjoint $\Lambda\cX$;
  \item $K:=PSU_n$ on linearly-independent, self-adjoint $\Lambda\cX$;
  \item $K$ on spans $\langle \cX\rangle$ for linearly-independent, self-adjoint $\cX$
  \end{itemize}
  all have Zariski-principal orbits for arbitrary $d$ and $n$.
\end{corollary}
\begin{proof}
  This is immediate from \Cref{th:gencj}: in each case pick a Zariski-open non-empty (and hence dense, by the irreducibility of the varieties involved) set where the isotropy group has minimal dimension and the smallest number of components in the usual topology.
\end{proof}

\begin{remark}
  Actions of compact Lie groups on connected manifolds always have principal orbits in the usual topology ( as proved for instance in \cite[Proposition I.2.5]{aud} and noted in \cite[Introduction]{rich-zero}); in \Cref{cor:princorb}, however, we are interested in the Zariski topology instead.
\end{remark}

In light of \Cref{cor:princorb} the following piece of notation and terminology makes sense.

\begin{notation}\label{not:genk}
  For positive integers $d$ and $n$ we write $\cat{gen} K_{\Lambda\cX}(d,n)$ of $\cat{gen} K_{\langle\cX\rangle}(d,n)$ for the isotropy group in $K:=PSU_n$ of a generic $\Lambda\cX$ or, respectively, $\langle \cX\rangle$ for self-adjoint, traceless, linearly-independent $d$-tuples $\cX\subset M_n$. Here, `generic' means ranging over an open dense subspace of the respective variety (Grassmannian in one case, space of tuples in the other).

  This is not quite a subgroup of $K$ but rather a conjugacy class therein; we abuse notation slightly by speaking of $\cat{gen} K_{\cdots}$ as groups. We also suppress $d$ and/or $n$ from the notation when convenient, as in $\cat{gen} K$ or $\cat{gen} K(d)$.
\end{notation}

\begin{remark}\label{re:gentrivsign}
  The triviality of $\cat{gen} K_{\Lambda\cX}(d,n)$ says that the $d$-tuples $\langle \cX\rangle$ representing quantum graphs with trivial automorphism groups form a ``large'' space: the complement of that space is a less-than-full-dimensional variety, and in particular that complement has measure zero; hence the relevance of (the triviality of) $\cat{gen} K_{\Lambda\cX}(d,n)$ to \Cref{th:new5}. 
\end{remark}

In the following statement a group-theoretic property $\cP$ (e.g. being finite, trivial, abelian, metabelian, solvable, nilpotent, etc. etc.) is
\begin{itemize}
\item {\it hereditary} if property $\cP$ for a group $\Gamma$ entails the property for all subgroups of $\Gamma$;
\item {\it algebraic} if property $\cP$ for a subgroup of a linear algebraic group entails it for its Zariski closure.
\end{itemize}

\begin{corollary}\label{cor:genfix}
  Let $G:=PSL_n$ and $K:=PSU_n$ be as in \Cref{th:gencj} and $\cP$ a hereditary group-theoretic property for subgroups of a complex linear algebraic group.
  \begin{enumerate}[(1)]
  \item\label{item:6} If the isotropy group $K_{\Lambda\cX}$ has property $\cP$ for at least one linearly independent self-adjoint tuple $\cX$ as in \Cref{eq:fam}, then this is the case for a Zariski-open dense set of tuples $\cX$.
  \item\label{item:7} The analogous statement holds for isotropy groups $K_{\langle\cX\rangle}$ of spans. 
  \item\label{item:8} If furthermore $\cP$ is algebraic then under the hypothesis of \Cref{item:6} its conclusion holds for algebraic isotropy groups $G_{\Lambda\cX}$.
  \end{enumerate}
\end{corollary}
\begin{proof}
  \Cref{item:6,item:7} follow immediately from parts \Cref{item:4,item:5} of \Cref{th:gencj}, noting that all varieties involved are irreducible and hence Zariski-density follows from openness and non-emptiness. As for \Cref{item:8}, it follows from the remark made in the course of the proof of \Cref{th:gencj} that $K_{\Lambda\cX}\subset G_{\Lambda\cX}$ is maximal compact (and hence Zariski-dense, since $G_{\Lambda\cX}$ is reductive by \Cref{le:fixisred}).
\end{proof}

\subsection{Small isotropy groups}\label{subse:sm}

We will be interested in proving results to the effect that the isotropy unitary group $PSU_{n,\langle\cX\rangle}$ is small, generically in the self-adjoint tuple $\cX$. We will make repeated, tacit use of \Cref{cor:genfix}, often reducing this to proving triviality for a {\it single} tuple $\cX$. A number of other simplifying assumptions will frequently be in place:

\begin{itemize}
\item since $1\in M_n$ is of course fixed under conjugation, we may as well focus on {\it traceless} $\cX$ (see \Cref{cv:perm}); 
\item since $d$-tuples of traceless $n\times n$ matrices generically have maximal dimension
  \begin{equation*}
    \min(d,n^2-1),
  \end{equation*}
  we assume $d\le n^2-1$ and focus on linearly independent tuples.
\item A single self-adjoint matrix will of course have a positive-dimensional commutant in $M_n$, so for $PSU_{n,\langle\cX\rangle}$ to ever be trivial we need $2\le d\le n^2-1$. Additionally, \Cref{lem:ortho} below also requires $(n^2-1)-d\ge 2$ in order to have trivial isotropy groups. In short, we need 
  \begin{equation*}
    2\le d\le n^2-3,
  \end{equation*}
  so we focus on $n\ge 3$ to ensure this range is non-empty.
\end{itemize}

We would like to mention here that the case of a single self-adjoint matrix is the only one in which the commutant will be non-trivial. To wit, two generic self-adjoint matrices generate the full matrix algebra; this fact is well-known but we could not locate a precise reference, so we add a proof.

\begin{lemma}\label{lem:twogenerate}
  The set of pairs $(A,B)$ of $n\times n$ Hermitian matrices such that the $*$-algebra by generated by $A$ and $B$ is not equal to $M_n$ is contained in a proper Zariski-closed subset of $M_{n,sa}^2$, and hence in particular has measure zero.
\end{lemma}
\begin{proof}
  The condition that $A,B\in M_{n,sa}$ generate $M_n$ as a complex $*$-algebra is equivalent to requiring that the commutators of $A$ and $B$ intersect only along the scalars, i.e. that the intersection of those commutators be minimal-dimensional.
  
  Denote by $M_{n,ssa}\subset M_{n,sa}$ the subset of self-adjoint matrices with simple spectrum. $M_{n,ssa}^2\subset M_{n,sa}^2$ is Zariski-dense, so we may as well work with simple-spectrum $A$ and $B$. in that case, the commutators are precisely the algebras $\langle A\rangle$ and $\langle B\rangle$, with respective bases $\{A^i\}_{i=0}^{n-1}$ and similarly for $B$.

  Now, the condition that
  \begin{equation*}
    \mathrm{span}\{A^i,\ 0\le i\le n-1\}\quad\text{and}\quad \mathrm{span}\{B^i,\ 0\le i\le n-1\}
  \end{equation*}
  intersect only along $\bC 1$ is Zariski-open, hence the conclusion.
\end{proof}

\begin{lemma}\label{lem:ortho}
  If $\cX$ is a linearly independent, self-adjoint traceless $d$-tuple then
  \begin{equation*}
    PSU_{n,\langle\cX\rangle} = PSU_{n,\langle\cX\rangle^{\perp}},
  \end{equation*}
  where `$\perp$' indicates the orthogonal complement with respect to the trace inner product
  \begin{equation*}
    \braket{A|B} = \mathrm{tr}(A^*B).
  \end{equation*}
\end{lemma}
\begin{proof}
  This is immediate: the conjugation action by $U_n$ is unitary with respect to the trace inner product in the statement.
\end{proof}

It will be convenient in the sequel to record the following simple observation, for future reference.

\begin{lemma}\label{le:auxrep}
  Let $G$ be a compact Lie group and $V$ a complex $G$-representation. Fix a subspace $W\subseteq V$, and denote, in general, by $K_U\subseteq G$ the closed subgroup leaving a subspace $U\subseteq V$ invariant.

  If $K_W$ acts trivially on $W$ then the same holds for $W'$ ranging over a neighborhood of $W$ in the Grassmannian $\mathrm{Gr}:=\mathrm{Gr}(\dim W,V)$.
\end{lemma}
\begin{proof}
  We will prove the contrapositive. Considering a convergent sequence $W_n\to W$ in $\mathrm{Gr}$, and suppose each $K_{W_n}$ acts {\it non}-trivially on the respective $W_n$. This means, in particular, that some $g_n\in K_{W_n}$ has an eigenvalue $\lambda_n\ne 1$ on $W_n$, and raising $g_n$ to some power if necessary we can assume that all $\lambda_n$ lie outside a fixed open neighborhood $U$ of $1\in \bS^1$.

  But then, because
  \begin{itemize}
  \item $G$ is compact, 
  \item the condition that an element of $G$ leave $W'\in\mathrm{Gr}$ invariant is closed in $G\times\mathrm{Gr}$ and
  \item $\bS^1\setminus U$ is closed,
  \end{itemize}
  we can find some subsequence of $x_n$ converging to $x\in K_W$ with an eigenvalue $\lambda\in \bS^1\setminus U$. on $W$. This contradicts the triviality of the action of $K_W$ on $W$, finishing the proof.
\end{proof}

Next, a number of inductive results. 

\begin{proposition}\label{pr:indstp}
  Fix $n\ge 3$ and assume that for any $d\in \{2,3,\dots,n^2-3\}$ the conjugacy class $\cat{gen} K_{\langle \cX\rangle}(d,n)$ is trivial. Then $\cat{gen} K_{\langle \cX\rangle}(d',n+1)$ is trivial for $2\le d'\le n^2-2$. 
  
  The same holds for abelian $\cat{gen} K$ or for the requirement that it have order $\le 2$.
\end{proposition}
\begin{proof}
  We focus on the first statement, the second one being entirely analogous. Consider two cases:

  {\bf ($2\le d\le n^2-3$)} Consider a generic self-adjoint linearly-independent $d$-tuple $(X_i)_{i=1}^d$ of $(n+1)\times (n+1)$ matrices all of whose non-zero entries are concentrated in the upper left-hand $n\times n$ corner. The non-unital subalgebra of $M_{n+1}$ it generates is the upper-left-hand corner $M_n\subset M_{n+1}$, so any unitary $U$ that preserves the span $\langle \cX\rangle$ under conjugation will also preserve that matrix algebra. But then it follows that $U$
  \begin{itemize}
  \item decomposes as a block-diagonal matrix with blocks of size $n$ and $1$,
  \item and hence its upper left-hand $n\times n$ block acts by conjugation on $M_n\subset M_{n+1}$ preserving $\langle \cX\rangle$ therein.
  \end{itemize}
  The hypothesis and the genericity of $\cX$ imply that the upper left-hand $n\times n$ block of $U$ is scalar. This is true of {\it every} $U$ projecting to $K_{\langle \cX\rangle}$, and hence $K_{\langle \cX\rangle}$ acts trivially on $\langle \cX\rangle$ (i.e. fixes every vector, not just the span). By \Cref{le:auxrep} the same holds of all tuples sufficiently close to $\cX$; since generically, $\cX\subset M_{n+1}$ generates the algebra $M_{n+1}$ (\Cref{lem:twogenerate}), it follows that $K_{\langle \cX\rangle}$ must indeed be trivial.
  
  As for the remaining case:

  {\bf ($d=n^2-2$)} Simply repeat the argument above appending one additional diagonal matrix to the tuple, with non-zero $(n+1)\times (n+1)$ entry.  
\end{proof}

Also note the following variant.

\begin{lemma}\label{le:indstp}
  Suppose that for some $n\ge 3$ and $1\le d$ there is a self-adjoint $d$-tuple $\cY$ of $(n-1)\times (n-1)$ matrices with zero diagonal such that no non-scalar diagonal unitary preserves their span. Then, $\cat{gen} K_{\langle \cX\rangle}(d+1,n)$ is trivial.
\end{lemma}
\begin{proof}
  As in the proof of \Cref{pr:indstp}, regard the matrices in $\cY$ as elements of $M_n$ via the upper-left-corner embedding $M_{n-1}\subset M_n$. Then, expand $\cY$ to $\cX:=\cY\cup \{X\}$ by adding a generic diagonal self-adjoint traceless operator $X\in M_n$. 

  The algebra generated by $\cX$ will be $M_{n-1}\times \bC\subset M_n$, so the isotropy group $G=K_{\langle \cX\rangle}$ preserving $\langle \cX\rangle$ embeds in $U_{n-1}\times \bS^1\subset U_{n}$ and hence leaves invariant the decomposition $\bC^n\cong \bC^{n-1}\oplus \bC$ obtained by separating the last coordinate. $G$ thus acts on the span $\cX'\subset M_{n-1}$ of the top left $(n-1)\times(n-1)$ blocks of the matrices in $\cX$.

  The trivial-diagonal assumption ensures that the upper left block $X'$ of $X$ is the unique operator in $\langle\cX'\rangle$
  \begin{itemize}
  \item in its conjugacy class, and    
  \item having trace $\mathrm{tr}~X'$,
  \end{itemize}
  since any such operator is of the form $X'+A$ for off-diagonal $A$ and hence being conjugate to $X'$ entails $A=0$ by a Hilbert-Schmidt-norm comparison.

  It follows that $G$ fixes $X'$ and hence consists of diagonal operators. But then the additional assumption on $\cY$ shows that $G$ is a scalar on $\bC^{n-1}$. All in all, this means that $G$ acts trivially on $\cX$; we can now conclude as in the proof of \Cref{pr:indstp}, via \Cref{le:auxrep}.
\end{proof}


We will use \Cref{le:indstp} in conjunction with the following simple observation.

\begin{lemma}\label{le:nodiag}
  Let $p\ge 2$ be a positive integer. A generic $m$-dimensional real subspace $W\le \bC^p$ with $1\le m\le 2p-1$ is not invariant under any non-trivial diagonal unitary $\pm1 \ne U\in U_p$.
\end{lemma}
\begin{proof}
  Denote by $e_i$, $1\le i\le p$ the standard basis of $\bC^p$ and identify
  \begin{equation*}
    V_i:=\bC e_i\cong \bR^2
  \end{equation*}
  (and thus $\bC^p\cong \bR^{2p}$). For each $1\le i\le p$ we can find $m$-dimensional subspaces of $\bR^{2p}$ that intersect $V_i$ along a line. Such spaces cannot be invariant under diagonal unitaries whose $i^{th}$ eigenvalue is $\ne \pm 1$. This non-invariance is generic for each $i$, and hence generic $W$ will not be invariant under any diagonal non-involutive unitaries.

  This leaves the diagonal involutions to deal with, i.e. the diagonal unitaries with eigenvalues $\pm 1$. If such a unitary $U$ is non-scalar then it splits $\bR^{2p}$ into a direct sum of non-zero $\pm 1$ eigenspaces $V_{+}$ and $V_{-}$, and $W$ is $U$-invariant if and only if
  \begin{equation*}
    W = (W\cap V_+)\oplus (W\cap V_-).
  \end{equation*}
  That this is generically not the case can be verified by a simple dimension count: there are only finitely many choices for $V_{\pm}$ (aligned with the fixed coordinate system we are considering), and one can check that for each specific choice of non-zero $V_{\pm}$ and each choice of
  \begin{equation*}
    \dim W\cap V_+,\quad \dim W\cap V_-
  \end{equation*}
  (of which there are finitely many) the dimension of
  \begin{equation*}
    (\text{variety of possible }W\cap V_+)\times (\text{variety of possible }W\cap V_-)
  \end{equation*}
  is smaller than that of the real Grassmannian $Gr(m,2p)$ parametrizing our $W$.
\end{proof}

The same type of dimension-counting argument used in the proof of \Cref{le:nodiag} also proves the following variant, of use below.

\begin{lemma}\label{le:pln}
  The conclusion of \Cref{le:nodiag} holds for real planes $W\le \bC^p$, $\dim_{\bR}W=2$ and unitaries $U$ ranging over the semidirect product $D_p\rtimes S_p$ of diagonal matrices by permutation matrices. 
\end{lemma}
\begin{proof}
  \Cref{le:nodiag} shows that a unitary $U\in D_p\rtimes S_p$ leaving a generic $W$ invariant must have finite order, bounded by $2p!$: indeed, if $U=DP$ is a decomposition as a product of a diagonal and a permutation matrix, then $U^t=\pm 1$ if $t$ is the least common multiple of the cycle lengths of $P$. Furthermore, by \Cref{le:nodiag} again, it is enough to consider factorizations $U=DP$ for non-trivial $P$ (otherwise we are back in the diagonal case).
  
  Decomposing $U$ as a product $DP$ for diagonal $D$ and a permutation matrix $P$ we will, for the sake of simplicity, assume that $P$ is a single cycle of order $p\ge 2$. In general it will decompose as a product of disjoint cycles, but this will suffice to illustrate the dimension count. Moreover, taking a shorter cycle would restrict the number of degrees of freedom even further, so the case of a cycle of order $p$ is the most demanding to rule out.

  We have $U^{p}=\pm 1$ (as noted above), so the $p$ eigenvalues of $D$ provide $p-1$ (real, as opposed to complex) degrees of freedom (because their product is $\pm 1$). $\bC^p\cong \bR^{2p}$ decomposes as a direct sum of $p$ 2-dimensional $U$-eigenspaces, and for $W$ invariance under $U$ means:
  \begin{itemize}
  \item coinciding with one of these eigenspaces, which makes for a 0-dimensional variety of such $W$ for each choice of $D$-eigenvalues, or
  \item decomposing as the direct sum of two lines, one in each of two distinct $U$-eigenspaces; this makes for a 2-dimensional variety, since each line is selected from a 2-plane and hence ranges over a real projective line. 
  \end{itemize}
  All in all, counting the $p-1$ degrees of freedom in selecting the eigenvalues of $D$, the two cases give spaces of dimension $p-1$ and $p+1$ respectively. In summary, the $W$ invariant under some $U=DP$ range over a $(p+1)$-dimensional space. On the other hand the Grassmannian parametrizing the planes $W\le \bR^{2p}$ has dimension $2(2p-2)>p+1$ (since $p\ge 2$), finishing the proof.
\end{proof}


\begin{corollary}\label{cor:n1n2}
  For $3\le n$ and $2\le d\le (n-1)(n-2)$ the group $\cat{gen} K_{\langle \cX\rangle}(d,n)$ is trivial.  
\end{corollary}
\begin{proof}
  This follows from \Cref{le:indstp} (with $d$ in place of that result's $d+1$), noting that by \Cref{le:nodiag} applied to
  \begin{align*}
    \bC^p &\cong \text{strictly upper-triangular $(n-1)\times (n-1)$ matrices}\\
          &\cong \text{off-diagonal self-adjoint $(n-1)\times (n-1)$ matrices}
  \end{align*}
  we can find
  \begin{equation*}
    1\le d-1\le 2p-1 = (n-1)(n-2)-1
  \end{equation*}
  self-adjoint off-diagonal $(n-1)\times (n-1)$ matrices whose span is not preserved by non-scalar unitaries. By \Cref{le:indstp}, this gives us the desired range $2\le d\le (n-1)(n-2)$.
\end{proof}

\begin{convention}\label{cv:abuse}
  Throughout the discussion below, we will abuse language slightly by referring to elements of $K_{\langle \cX\rangle}$ as matrices (rather than matrices modulo scalars).
\end{convention}

\begin{theorem}\label{th:new5}
  For $3\le n$ and $2\le d\le n^2-3$ the group $\cat{gen} K_{\langle \cX\rangle}(d,n)$ is trivial.
\end{theorem}
\begin{proof}
  We treat the case $n\ge 5$ here, delegating the small cases $n=3,4$ to \Cref{pr:34}. The argument is by induction on $n\ge 5$, via \Cref{pr:indstp}.

  In the base case $n=5$ \Cref{cor:n1n2} takes care of the range $2\le d\le (n-1)(n-2)=12$. Since this is precisely half of the range $2..(n^2-1=24)$, we conclude by \Cref{lem:ortho}, which allows us to bootstrap the left-hand half.

  As for the induction step of passing to $n+1$ from lower values, we obtain
  \begin{itemize}
  \item the range $2..n^2-2$ from \Cref{pr:indstp};
  \item the range $2..2n+1$ by applying \Cref{cor:n1n2} to $n+1$ (in place of $n$) and noting that
    \begin{equation*}
      n(n-1)=((n+1)-1)((n+1)-2)
    \end{equation*}
    dominates our desired upper bound $2n+1$ because $n\ge 5$ (in fact $n\ge 4$ would suffice here);
  \item and hence the range
    \begin{equation*}
      n^2-1\quad ..\quad (n+1)^2-3\quad =\quad ((n+1)^2-1)-(2n+1)\quad ..\quad ((n+1)^2-1)-2
    \end{equation*}
    by reflection, via \Cref{lem:ortho}.
  \end{itemize}
  This finishes the proof.
\end{proof}

\begin{corollary}\label{cor:genab}
  For $3\le n$ and $1\le d\le n^2-2$ the group $\cat{gen} K_{\langle \cX\rangle}(d,n)$ is abelian.  
\end{corollary}
\begin{proof}
  The case $2\le d\le n^2-3$ is covered by \Cref{th:new5}, whereas a {\it single} traceless matrix $X$ will have simple spectrum and be inequivalent to $-X$, ensuring that the subgroup of $U_n$ preserving the span $\bR X$ is diagonal in a basis diagonalizing $X$.
\end{proof}

\begin{proposition}\label{pr:34}
  The conclusion of \Cref{th:new5} holds for $n=3,4$. 
\end{proposition}
\begin{proof}
  Given the symmetry $d\leftrightarrow n^2-1-d$ provided by \Cref{lem:ortho}, \Cref{cor:n1n2} covers everything except for the cases
  \begin{equation*}
    (n,d) = (3,3),\ (3,4)\text{ and }(4,7),
  \end{equation*}
  which we treat here. 

  $(n,d)=(3,3):$ In general, for any $n\ge 3$, we can resolve the $d=n$ case by selecting a tuple $\cX$ consisting of $n-1$ linearly independent diagonal matrices $X_1$ up to $X_{n-1}$ together with an $n^{th}$ matrix $X_n$ that
  \begin{itemize}
  \item commutes with no non-scalar diagonal operators, ensuring that every $U\in K_{\langle \cX\rangle}$ preserves the diagonal algebra and hence belongs to the semidirect product $D_n\rtimes S_n$ of the diagonal unitaries and the permutation matrices;
  \item is not unitarily equivalent to $-X_n$, so that a unitary that preserves the span $\bR X_n$ must commute with $X_n$;
  \item is generic enough to ensure that no non-scalar unitary commutes with it, as in the proof of \Cref{cor:n1n2} (via \Cref{le:nodiag});
  \item has above-diagonal entries of distinct moduli, ensuring that no $DP\in D_n\rtimes S_n$ for a non-trivial permutation $P$ commutes with it.
  \end{itemize}
  Such a tuple $\cX$ would then have trivial $K_{\langle \cX\rangle}\subset PSU_n$.

  $(n,d)=(3,4)$ or $(4,7):$ We treat the cases in parallel to an extent. In both cases, we first consider tuples $\cX$ constructed as follows:
  \begin{itemize}
  \item $n-1$ traceless diagonal operators $X_i$, $1\le i\le n-1$ spanning (together with $1$) the diagonal algebra $D_n$;
  \item a linearly independent, self-adjoint traceless tuple $\mathcal{ND}$ of $2n-4$ operators $X_i$, $n\le i\le d$, all of whose non-zero entries are non-diagonal and on the last row/column.
  \end{itemize}
  If $\mathcal{ND}$ is chosen generically so as to ensure the span $\langle\mathcal{ND}\rangle$ contains no matrices with precisely {\it one} non-zero entry on the last column, the only elements in the span $\langle \cX\rangle$ whose squares belong to $\langle 1,\cX\rangle$ are diagonal. This intrinsic characterization of $D_n\subset \langle \cX\rangle$ implies that $K_{\langle \cX\rangle}$ preserves it, and hence once more
  \begin{equation*}
    K_{\langle \cX\rangle}\subseteq DU_n\rtimes S_n.
  \end{equation*}
  On the other hand $K_{\langle \cX\rangle}$ also preserves the orthogonal complement $\langle \mathcal{ND}\rangle$ of $D_n$ in $\langle 1,\cX\rangle$, so in fact it must be contained in the semidirect product $DU_n\rtimes S_{n-1}$, where the latter is the group of $n\times n$ permutation matrices permuting only the first $n-1$ coordinates.

  Next, note that the real subspace $W$ of $\bC^{n-1}$ spanned by the last columns of the matrices in $\mathcal{ND}$ (where the last columns are regarded as having $n-1$ entries upon dropping their last zero entry) has real dimension $2n-4$. This means that by \Cref{le:pln}, generically, that span is not preserved by any operator in $DU_{n-1}\rtimes S_{n-1}$ save for $-I_{n-1}$: for $n=3$ we can apply \Cref{le:pln} directly, whereas for $n=4$ we can apply it to the 2-dimensional orthogonal complement of the 4-dimensional $W$. This in turn means that the only unitary operators in $DU_n\rtimes S_{n-1}$ that preserve $\langle \mathcal{ND}\rangle$ are those of the form
  \begin{equation}\label{eq:one-}
    \mathrm{diag}(\lambda,\cdots,\lambda,\pm\lambda),\quad \lambda\in\bS^1. 
  \end{equation}
  In short, $K_{\langle \cX\rangle}$ must be contained in this latter diagonal group.

  Now select new tuples repeating the construction virtually verbatim, the only difference being that the $2n-4$ matrices (for $n=3,4$) in $\mathcal{ND}$ now have non-zero entries only immediately above/below the main diagonal (i.e. entries $(i,i\pm 1)$ for $1\le i\le n-1$).
 
  Essentially the same argument this time around shows that $K_{\langle \cX\rangle}$ will be contained in $D_n\rtimes S_n$ and act as $-1$ on the off-diagonal matrices $\mathcal{ND}$, But this means that (unless scalar) it consists of diagonal unitary matrices of the form
  \begin{equation}\label{eq:alt-}
    (\lambda,-\lambda,\cdots, -\lambda,\lambda)
  \end{equation}
  (alternating signs).

  This already settles the case $n=4$: if $\cat{gen} K_{\langle \cX\rangle}(7,4)$ were non-trivial then \Cref{eq:one-,eq:alt-}, which are obtainable by limiting the generic behavior, would mean that generically the non-trivial element of $K_{\langle \cX\rangle}$ has (up to scaling) both a 1 and a 2-dimensional $(-1)$-eigenspace: a contradiction.

  As for the remaining case $(n,d)=(3,4)$, we can take for $\cX$ a tuple consisting of {\it one} diagonal matrix $X_1$ together with 5 generically-chosen matrices as in previous $\mathcal{ND}$. $\langle \cX\rangle$ thus consists of matrices of the form
  \begin{equation*}
    \begin{pmatrix}
      x & 0 & a\\
      0 & y & b\\
      \overline{a} & \overline{b} & z
    \end{pmatrix}
  \end{equation*}
  with $x$, $y$ and $z$ unique up to simultaneous scaling and $(a,b)$ ranging over a 3-dimensional real subspace of $\bC^2$. If, say, $x$ and $y$ have equal signs and $(x,y,z)$ is otherwise generic, such a matrix will be invertible as soon as it has non-zero diagonal. This means that once more $K:=K_{\langle \cX\rangle}$ preserves the span $\bR X_1$ of the diagonal matrix $X_1=\mathrm{diag}(x,y,z)$ and hence also $\langle \mathcal{ND}\rangle$. The genericity of $X_1$ further ensures that $K$ consists of diagonal matrices.

  But according to the discussion above, if the generic $K$ is non-trivial then it consists (modulo scalars) of involutive matrices with 1-dimensional $(-1)$-eigenspace (cf. \Cref{eq:one-}), acting on $\langle \cX\rangle$ with 2-dimensional $\pm 1$-eigenspaces. It remains to observe that in the present setup no diagonal unitary with 1-dimensional $(-1)$-eigenspace can act with a 2-dimensional $(-1)$-eigenspace on the 3-dimensional real span consisting of 
  \begin{equation*}
    \begin{pmatrix}
      0 & 0 & a\\
      0 & 0 & b\\
      \overline{a} & \overline{b} & 0
    \end{pmatrix}
  \end{equation*}
  with $(a,b)$ as above. This contradicts the non-triviality of the generic $K$, finishing the proof.
\end{proof}

\subsection{The degree matrix}
In this subsection we take another route to the main results. We will employ properties of the degree matrix of a quantum graph. \Cref{thm:abelian} is the same as \Cref{cor:genab} and \Cref{thm:trivialauto} is a weaker version of \Cref{th:new5}. Even though we do not fully recover our results, in this approach we obtain somewhat explicit examples of tuples with trivial automorphism groups, where we only use a little randomness to obtain diagonal matrices with certain properties.
\begin{proposition}\label{Prop:gendegmat}
Let $V=\op{span}(\mathds{1}, X_1,\dots, X_d) \subset M_n$ be an operator system generated by $d$ traceless self-adjoint matrices $X_1,\dots X_d$. Let $A: M_n \to M_n$ be the corresponding quantum adjacency matrix, given by \Cref{Prop:quantumadj}. If the \textbf{degree matrix} $D:=A\mathds{1}$ has simple spectrum then the automorphism group of $V$ is abelian. Moreover, the complement of the set of tuples $(X_1,\dots, X_d)$ such that $D$ has simple spectrum is a closed subvariety, so if there exists such a $d$-tuple, almost surely all $d$-tuples have this property.
\end{proposition}
\begin{proof}
If $U$ is a unitary in $M_n$ such that $U V U^{\ast}=V$, then $U A(U^{\ast}xU) U^{\ast} = A(x)$, hence $U(A\mathds{1}) U^{\ast} = A(\mathds{1})$. So, as before, if $D=A\mathds{1}$ has simple spectrum then the set of such unitaries is commutative.

By \Cref{Prop:quantumadj} we have $D = n\sum_{i} A_{i}^2$, if $(A_i)$ is an orthonormal basis of $V$ consisting of Hermitian matrices. Such a basis can be obtained from the tuple $(X_1,\dots, X_d)$ via the Gram-Schmidt procedure, i.e. using only algebraic operations. Since having simple spectrum can also be described as vanishing of a certain polynomial, the set of ``bad'' tuples is algebraic, so it has measure zero as soon as it is proper.
\end{proof}
We will be working with operator systems for which the degree matrix $D$ is diagonal. We need the following lemma to show that one can construct operator systems with diagonal degree matrices that have no repeated eigenvalues.
\begin{lemma}\label{lemma:diag}
Let $1\leqslant d \leqslant n-2$. Let $\Lambda = (\lambda_1,\dots, \lambda_n)$ be an arbitrary diagonal matrix. Then there exists a $d$-tuple of traceless diagonal matrices $(X_1,\dots,X_d)$ such that $D+\Lambda$ has no repeated eigenvalues.
\end{lemma}
\begin{proof}
  Let $U\in O(n)$ be a Haar orthogonal matrix. Then for each $j$ we can consider the diagonal matrix $u_j:= \sum_{k} u_{jk} e_{k}$, where $e_k$ is a shorthand for $e_{kk}$. For different $j$'s these matrices are mutually orthogonal. Let $(\tilde{u}_{j})_{j\in [n]}$ be the centered version of $(u_j)_{j\in[n]}$, i.e. we subtract from each $u_j$ a suitable multiple of the identity to make it traceless and then apply the Gram-Schmidt procedure to obtain an orthonormal set. Then $V:= \op{span}(\mathds{1}, \tilde{u}_1,\dots,\tilde{u}_d)$ is an operator system whose degree matrix is equal to $D=\mathds{1} + n\sum_{j=1}^{d} \widetilde{u}_j^2$. For convenience, we will work with the matrix $\widetilde{D}:= \sum_{j=1}^{d} \widetilde{u}_{j}^2$.

We want to show that there exists a choice of $U$ such that $\widetilde{D}+\Lambda$ has no repeated eigenvalues, i.e. all its entries are distinct. In fact it holds almost surely. In order to prove that, we have to show for a given pair of diagonal entries that they are different almost surely; it will follow that almost surely all entries are different.

Let us do it for entries $1$ and $2$. Clearly equality of these two entries is an algebraic equation in the entries of $U$, so it will hold on measure zero set as soon as we can exhibit a single example. We will start with the following vectors $u_1 = c_1 e_{1} -c_2 e_{d+1} - c_3 e_{d+2}$, $\widetilde{u}_k =  e_{k} - \frac{1}{n} \sum_{l=1}^{n} e_l$ for $k\in \{2,\dots,d\}$. All the $\widetilde{u}_k$'s for $k\geqslant 2$ are orthogonal to $u_1$, so the same will be true for $(u_k)_{k=2}^{d}$ -- the orthonormal family obtained from $(\widetilde{u}_k)_{k=2}^{d}$ via the Gram-Schmidy procedure. We need the coefficients of $u_1$ to satisfy $c_1=c_2+c_3$ and $c_1^2 + c_2^2 + c_3^2 = 1$. Note that the contribution to the entries of $\widetilde{D}+\Lambda$ coming from the vectors $(u_k)_{k=2}^{d}$ is independent of our choice for $u_1$, so we can include it in the matrix $\Lambda$, by forming a new matrix $\widetilde{\Lambda}$. With this choice the first entry of $\widetilde{D} + \widetilde{\Lambda}$ is equal to $c_1^2 + \widetilde{\lambda}_1$ and the second one is equal to $ \widetilde{\lambda}_2$. We can easily choose an appropriate $c_1$ so that the two are not equal. The lemma follows.
\end{proof}

\begin{theorem}\label{thm:abelian}
  Fix $n$ and $d \in \{1,\dots, n^2 - 2\}$. For almost every $d$-tuple $(X_1,\dots, X_d)$ of traceless self-adjoint matrices the automorphism group of the quantum graph $V:= \op{span}\{\mathds{1}, X_1,\dots, X_d\}$ is abelian.
\end{theorem}
Before proving this theorem, we need to introduce a certain orthonormal family of Hermitian matrices in $M_n$ that will allow us to build operator systems with diagonal degree matrices. For any pair $k\neq l$ of numbers in $\{1,\dots,n\}$ we define
\begin{equation*}
  f_{kl}:= \left\{ \begin{array}{cl} \frac{1}{\sqrt{2}} (e_{kl} + e_{lk}) &\text{if } k<l \\ \frac{i}{\sqrt{2}} (e_{kl} - e_{lk}) &\text{if }k>l \end{array}\right.
\end{equation*}
Note that the family $\mathcal{F}:=(f_{kl})$ is orthonormal and $f_{kl}^2 = \frac{1}{2} (e_{kk} + e_{ll})$ is diagonal.
\begin{proof}[of \Cref{thm:abelian}]
By \Cref{Prop:gendegmat}, for any $d\in \{1,\dots, n^2-2\}$ we just need to provide a single tuple such that the degree matrix has no repeated eigenvalues.

If $d\leqslant n-2$ then we can use the \Cref{lemma:diag} with $\Lambda=0$. If $d>n-2$, then we take a subset of the family $\mathcal{F}$ of cardinality $d-n+2$, which is always possible, because $\mathcal{F}$ has $n^2-n$ elements. We get a corresponding degree matrix $D_1$, which is diagonal. By \Cref{lemma:diag} we can choose an $n-2$-tuple of traceless diagonal matrices with the corresponding diagonal degree matrix $D_2$ such that $D_1+D_2$ has no repeated eigenvalues. But $D:=D_1+D_2$ is exactly the degree matrix of the combined operator system, which has therefore a degree matrix with simple spectrum. This ends the proof of the theorem.
\end{proof} 

\begin{theorem}\label{thm:trivialauto}
  Let $n\geqslant 6$ and $4\le d\le n^2-5$. We can construct a $d$-tuple $(X_1,\dots , X_d)$ of traceless Hermitian matrices, such that the degree matrix is diagonal and has simple spectrum and the automorphism group of the corresponding quantum graph is trivial.
\end{theorem}
\begin{proof}

We will first show that there exists an appropriate family of Hermitian matrices with zero diagonals and with a diagonal degree matrix such that there is no nontrivial diagonal matrix that preserves the span upon conjugation. Then we will invoke the \Cref{lemma:diag} to upgrade our degree matrix to a one with simple spectrum, which will show that diagonal matrices were the only possible candidates for automorphisms.

Assume first that $n\geqslant 7$. Define $X_{1} := \sum_{i=1}^{\left[(n-1)/2\right]} f_{2i,2i+1}$ and $X_2:= \sum_{i=1}^{\left[n/2\right]} f_{2i-1,2i}$. As $X_1$ and $X_2$ are orthogonal and $X_1^2$ and $X_2^2$ are diagonal, the corresponding degree matrix is diagonal. Let $U:=(u_1,\dots,u_n)$ be a diagonal unitary matrix. Since the automorphism group is really the quotient of the subgroup of unitaries by the center, i.e. the scalar matrices, we may assume that $u_1=1$. Recall that if $X=(x_{ij})$ is a matrix then $UXU^{\ast}$ has entries $(u_i u_j^{-1} x_{ij})$. We can now check when the $\op{span}(X_1,X_2)$ is preserved upon conjugation by $U$. One condition is that the entries $(i,i+1)$ and $(i+1,i)$ are equal which gives $u_{i} u_{i+1}^{-1} = u_{i+1} u_{i}^{-1}$, i.e. $u_{i}^2 = u_{i+1}^2$, i.e. the squares of entries are constant. As $u_1=1$, we conclude that $u_i =\pm 1$, so we will from now on write $u_i u_{i+1}$ instead of $u_{i} u_{i+1}^{-1}$. From the form of $X_1$ and $X_2$ we get that $u_{i} u_{i+1} = u_{i+2}u_{i+3}$. In particular, only $u_2$ and $u_3$ are free variables, because, for example, $u_4 = u_2 u_3$, since $u_1 u_2 = u_3 u_4$. From the equality $u_2 u_3 = u_4 u_5$ it follows that $u_5=1$. From $u_3u_4 = u_5u_6$, we get $u_6=u_3 u_4 = u_2$. Moreover, from $u_2 u_3=u_4 u_5 = u_6 u_7$ we get that $u_7=u_3$. We need to add another matrix to conclude that $u_2=u_3=1$, which would prove that $U=\mathds{1}$.

Our choice will be $Y = f_{14} + f_{25} + f_{37}$. We get $u_1u_4 = u_2 u_5 = u_3 u_7$. As $u_1=u_5=1$, we obtain $u_2=u_4$, so $u_3=1$, since $u_2 u_3 = u_4$. On the other hand, $u_3 u_7=1$, so $u_4=u_2=1$ and we conclude that all $u_i$'s are equal to one. So as soon as we have the span of $(X_1,X_2,Y)$ and nothing else happens at the entries used by these matrices we will not have nontrivial diagonal matrices preserving the subspaces upon conjugation. These matrices have in total $2n+4$ non-zero entries. We do not want to touch the diagonal, so we have $n^2 - 3n-4$ entries left, where we can insert other members of the family $\mathcal{F}= (f_{ij})$; whichever we choose, the resulting degree matrix will be diagonal. After having made this choice, we can invoke the \Cref{lemma:diag} to choose between $1$ and $n-2$ diagonal matrices that force the degree matrix of the whole quantum graph to have a simple spectrum.

We have to use at least $4$ matrices in this approach: $X_1$, $X_2$, $Y$, and one diagonal, which gives the lower bound $d\geqslant 4$. To get to $n^2-5$, note that it is sufficient to go up to $\frac{n^2-1}{2}$, by using \Cref{lem:ortho}. The only thing to note is that is that the degree matrix of the orthogonal complement will also be diagonal. This follows from the fact that the degree matrix of the sum of the two will be the sum of the degree matrices, and this sum is equal to the whole space of traceless matrices, whose degree matrix is a multiple of identity. Now note that in our construction we left $n^2-3n-4$ unused entries and we can add to it our $3$ matrices $X_1$, $X_2$, and $Y$, and up to $n-2$ diagonal matrices, so we we can have a $d$-tuple with $d=n^2-2n-3$. This is larger than $\frac{n^2-1}{2}$ already for $n\geqslant 3$.

To deal with the case $n=6$, we use slightly different matrices $X_1$, $X_2$, and $Y$. Namely, we take $X_1:=f_{12}+f_{34}+f_{56}$, $X_2:= f_{23} + f_{45} + f_{16}$, and $Y:= f_{15} + f_{36}$. By examining $X_1$ and $X_2$ we get the equalities $u_1 u_2 = u_3 u_4 = u_5 u_6$ and $u_2 u_3 = u_4 u_5 = u_1 u_6$. Remembering that $u_1=1$, we get $u_3=u_5=1$ and $u_2=u_4=u_6$. If we include $Y$, we get the additional condition $u_1 u_5 = u_3 u_6$, so $u_6=1$. In this case there are $16$ unused entries, so we can construct a $d$-tuple with $d= 23$, which is larger than $\frac{35}{2}$.
\end{proof}


\newcommand{\etalchar}[1]{$^{#1}$}
\def\polhk#1{\setbox0=\hbox{#1}{\ooalign{\hidewidth
  \lower1.5ex\hbox{`}\hidewidth\crcr\unhbox0}}}

\addcontentsline{toc}{section}{References}

\Addresses

\end{document}